\documentclass{amsart}

\usepackage{amssymb,amsmath,amsthm,amsaddr}
\usepackage{comment,enumerate,hyperref}

\newcommand{\alg}{\mathbf}

\newcommand{\set}[2]{\{ #1 \mid #2 \}}
\newcommand{\pair}[2]{\langle #1, #2 \rangle}

\newcommand{\sqleq}{\sqsubseteq}
\newcommand{\assign}{:=}
\newcommand{\unit}{\eta}

\newcommand{\dual}{\partial}
\newcommand{\bs}{\backslash}
\newcommand{\1}{1}
\newcommand{\0}{0}
\newcommand{\botbold}{\boldsymbol{\bot}}
\newcommand{\topbold}{\boldsymbol{\top}}
\newcommand{\intneg}{\mathop{\sim}}
\newcommand{\nagneg}{\mathop{\sim}}

\newcommand{\veegamma}{\vee_{\gamma}}
\newcommand{\cdotgamma}{\cdot_{\gamma}}
\newcommand{\bsgamma}{\bs\!\!\bs}
\newcommand{\sgamma}{/\!\!/}

\newcommand{\idmap}{\mathrm{id}}

\newcommand{\bsresright}{\makebox[8.5pt]{$\bs^{\scriptscriptstyle\!\!\ast}$}}
\newcommand{\bsresleft}{\makebox[8.5pt]{{$\kern 0.08667em_{\scriptscriptstyle\ast\!\!}\bs$}}}
\newcommand{\sresright}{\makebox[8.5pt]{{$/_{\scriptscriptstyle\!\!\ast}\kern 0.08667em$}}}
\newcommand{\sresleft}{\makebox[8.5pt]{$^{\scriptscriptstyle\ast\!\!} /$}}

\newcommand{\algL}{\alg{L}}
\newcommand{\algM}{\alg{M}}
\newcommand{\algN}{\alg{N}}
\newcommand{\algS}{\alg{S}}

\newcommand{\algNsigma}{\alg{N}_{\sigma}}
\newcommand{\algNgamma}{\alg{N}_{\gamma}}

\newcommand{\nagtimes}{\bowtie}

\newcommand{\algSmodM}{{\alg{S} \mathbin{\nagtimes} \alg{M}}}

\newcommand{\algSmodMzero}{{\alg{S} \mathbin{\nagtimes_{\0}} \alg{M}}}

\newcommand{\embed}{\iota}
\newcommand{\embedB}{\embed_{\scriptscriptstyle\alg{B}}}
\newcommand{\embedL}{\embed_{\scriptscriptstyle\alg{L}}}
\newcommand{\embedM}{\embed_{\scriptscriptstyle\alg{M}}}
\newcommand{\embedS}{\embed_{\scriptscriptstyle\alg{S}}}
\newcommand{\embedSplus}{\embed_{\scriptscriptstyle +}}
\newcommand{\embedSminus}{\embed_{\scriptscriptstyle -}}

\newcommand{\thistheoremref}{}

\newtheoremstyle{plainnewline}
{}{}
{\itshape}{}
{\bfseries}{.}
{\newline}{\thmname{#1}\thmnumber{ #2}\thistheoremref\thmnote{ (#3)}}

\theoremstyle{plainnewline}
\newtheorem{theorem}{Theorem}[section]
\newtheorem{proposition}[theorem]{Proposition}
\newtheorem{definition}[theorem]{Definition}

\theoremstyle{plain}
\newtheorem{fact}[theorem]{Fact}
\newtheorem{lemma}[theorem]{Lemma}

\newtheorem*{theorem*}{Theorem}
\newtheorem*{theoremA}{Theorem A}
\newtheorem*{theoremB}{Theorem B}

\author{Adam P\v{r}enosil}
\address{Universitat de Barcelona, Departament de Filosofia, Spain}
\email{adam.prenosil@ub.edu}
\author{Constantine Tsinakis}
\address{Vanderbilt University, Department of Mathematics, USA}
\email{constantine.tsinakis@vanderbilt.edu}

\title[Nagata products of bimodules]{Nagata products of bimodules \\ over residuated lattices}

\begin{document}

\begin{abstract}
  We study the (restricted) Nagata product construction, which produces a partially ordered semigroup from a bimodule consisting of a partially ordered semigroup acting on a (pointed) join semilattice. A canonical \mbox{example} of such a bi\-module is given by a residuated lattice acting on itself by division, in which case the Nagata product coincides with the so-called twist product of the residuated lattice. We show that, given some further structure, a pointed bimodule can be reconstructed from its restricted Nagata product. This yields an ad\-junction between the category of cyclic pointed residuated bimodules and a certain category of posemigroups with additional structure, which subsumes various known adjunctions involving the twist product construction.
\end{abstract}

\maketitle

\section{Introduction}


\newlength{\abovedisplayskipaux}
\setlength{\abovedisplayskipaux}{\abovedisplayskip}
\newlength{\belowdisplayskipaux}
\setlength{\belowdisplayskipaux}{\belowdisplayskip}
\setlength{\abovedisplayskip}{8pt}
\setlength{\belowdisplayskip}{8pt}

\newlength{\auxlength}
\newlength{\auxlengthtwo}
\settowidth{\auxlength}{$a \ast \0$}
\settowidth{\auxlengthtwo}{$\embedM$}

  The \emph{idealization} of a bimodule $\algM$ over a (commutative) ring $\alg{R}$ is a classical construction due to Nagata~\cite{nagata62}, which produces a (commutative) ring $\alg{I}$ where $\algM$ sits as a nilpotent ideal. The additive group of $\alg{I}$ is the direct sum of the additive groups of $\alg{R}$ and $\algM$ and multiplication is defined in $\alg{I}$ as
\begin{align*}
  \pair{a}{x} \cdot \pair{b}{y} & \assign \pair{a \cdot b}{a \ast y + x \ast b} \text{ for } a, b \in \alg{R} \text{ and } x, y \in \algM,
\end{align*}
  where $\ast$ denotes the left and right module action of $\alg{R}$ on $\alg{M}$. This is precisely the multiplication law obtained when one multiplies out $a+x$ and $b+y$ and decides to disregard the term $x \cdot y$. Alternatively, if one views $\pair{a}{x}$ and $\pair{b}{y}$ as formal fractions, then the above definition of multiplication in $\alg{I}$ becomes formally identical to the formula for computing the sum of two fractions:
\begin{align*}
  \frac{x}{a} + \frac{y}{b} = \frac{a \cdot y + x \cdot b}{a \cdot b}.
\end{align*}
  Because $\algM$ may be identified inside $\alg{I}$ with the ideal consisting of elements of the form $\pair{0}{x}$, the idealization construction allows one to extend results about ideals to modules (see~\cite{anderson+winders09} for more details).

  The above construction does not in fact require the full structure of a module over a ring. The structure of a module over a semiring, where the additive Abelian group structure in $\alg{R}$ and $\alg{M}$ is replaced by a commutative semigroup structure, suffices. Within the very broad class of modules over semirings, ordinary modules over rings form one extreme case, while the opposite extreme case replaces Abelian groups by \emph{idempotent} commutative semigroups (semilattices).

  Modules over idempotent semirings arise naturally in the algebraic study of non-classical logics, where \emph{residuated lattices} form a prominent class of algebras. These are algebras $\alg{L} \assign \langle L, \wedge, \vee, \cdot, \1, \bs, / \rangle$ such that $\langle L, \wedge, \vee \rangle$ is a lattice, $\langle L, \cdot, \1 \rangle$ is a monoid, and the two division operations are the so-called \emph{residuals} of multiplication, meaning that they satisfy (and are uniquely determined by) the residuation law
\begin{align*}
  b \leq a \bs c \iff a \cdot b \leq c \iff a \leq c / b.
\end{align*}
  Each residuated lattice $\alg{L}$ has an idempotent semiring reduct $\langle L, \vee, \cdot, \1 \rangle$, with $\vee$ playing the role of addition, which acts on the semilattice reduct $\alg{L}^{\dual} \assign \langle L, \wedge \rangle$ by division, i.e.\ it comes with the following left and right action of elements $a, b \in \algL$ on elements $x, y \in \algL^{\dual}$:
\begin{align*}
  a \ast y & \assign y / a, & x \ast b & \assign b \bs x.
\end{align*}

  Applying Nagata's idealization construction to the above bimodule induced by a residuated lattice $\alg{L}$ then equips the set $L \times L$ with the following operations, where $\vee$ plays the role of addition and $\circ$ is called the :
\begin{align*}
  \pair{a}{x} \vee \pair{b}{y} & \assign \pair{a \vee b}{x \wedge y}, \\
  \pair{a}{x} \circ \pair{b}{y} & \assign \pair{a \cdot b}{a \ast y + x \ast b} = \pair{a \cdot b}{y / a \wedge b \bs x}.
\end{align*}
  Here and throughout the paper, multiplication and division bind more tightly than meets and joins, so $y / a \wedge b \bs x$ is to be parsed as $(y / a) \wedge (b \bs x)$. Because $\alg{L}$ and $\alg{L}^{\dual}$ are in fact lattices, the join operation $\vee$ has a corresponding meet operation $\wedge$:
\begin{align*}
  \pair{a}{x} \wedge \pair{b}{y} & \assign \pair{a \wedge b}{x \vee y}.
\end{align*}
  Moreover, because the division action of $\alg{L}$ on $\alg{L}^{\dual}$ has the additional feature (defined more precisely below) that it is residuated, the multiplication $\circ$ has the following residuals with respect to the lattice order:
\begin{align*}
  \pair{a}{x} \bs \pair{b}{y} & \assign \pair{a \bs b \wedge x / y}{y \cdot a}, \\
  \pair{a}{x} / \pair{b}{y} & \assign \pair{a / b \wedge x \bs y}{b \cdot x}.
\end{align*}
  This yields the so-called \emph{twist product} $\alg{L}^{\bowtie}$ of a residuated lattice $\alg{L}$. (The name comes from the fact that the order in the second but not the first component is inverted in $\alg{L}^{\bowtie}$ compared to the product order $\alg{L} \times \alg{L}$.)

  The above twist product construction on residuated lattices and the module-theoretic perspective which underpins it were introduced by Tsinakis \& Wille~\cite{tsinakis+wille06}, who were inspired by the work of Chu on so-called $\ast$-autonomous categories and Girard quantales (see \cite{rosenthal90} and~\cite[Appendix]{barr79}). Special cases of this construction were already studied earlier, a prominent instance being twist products of Heyting algebras, which provide an algebraic semantics for Nelson's constructive logic with strong negation~\cite{odintsov08}.

  The aim of the present paper is to develop a module-theoretic approach to twist products from the ground up, starting from partially ordered semigroups (posemigroups). In its bare bones form, the \emph{Nagata \mbox{product}} construction takes a bimodule consisting of a posemigroup $\algS$ (written multiplicatively) acting on a posemigroup~$\algM$ (written additively) and produces a posemigroup $\algSmodM$ whose universe is $S \times M$ ordered component\-wise equipped with the Nagata multiplication. Throughout the paper, we adopt the useful convention of~\cite{tsinakis+wille06} that elements of $\algS$ (the scalars) will be denoted $a, b, c$ and elements of $\algM$ will be denoted $x, y, z$. This makes the notation $a \ast x$ and $y \ast b$ unambiguous: it always denotes $a$ acting on $x$ on the left and $b$ acting on $y$ on the right, respectively. (Note, however, that we also use the variables $x, y, z$ when writing down formal equational axioms.)

  The effect of the \mbox{Nagata} product is to package a two-sorted algebra consisting of two posemigroups linked by a module action into a single-sorted algebra. The main theme of our paper is the relationship between these two structures. We would like to be able to move back and forth between the two-sorted and single-sorted perspectives, ideally via an equivalence between some appropriate categories of two-sorted and single-sorted structures. More concisely, the central question is:
\begin{align*}
  \text{When and how can we reconstruct an $\algS$-bimodule $\algM$ from its Nagata product?}
\end{align*}

  In order to recover the original bimodule from its Nagata product, we will need to impose certain restrictions. Firstly, we shall require that $\algM$ be a join semilattice (i.e.\ both the semigroup reduct and the poset reduct of the posemigroup $\algM$ are the same join semilattice). Secondly, we shall require that $\alg{S}$ is a meet semilattice with a residuated semigroup multiplication and the action of $\algS$ on $\algM$ is a \emph{residuated action}. The latter requirement means that we have four additional division-like operations connecting $\algS$ and $\algM$ which satisfy the following residuation laws:
\begin{align*}
  a \leq y \sresleft x \iff a \ast x \leq y \iff x \leq a \bsresleft y, \\
  x \leq y \sresright a \iff x \ast a \leq y \iff a \leq x \bsresright y.
\end{align*}
  As a consequence, the Nagata product $\algSmodM$ is residuated, with
\begin{align*}
  \pair{a}{x} \bs \pair{b}{y} & \assign \pair{a \bs b \wedge x \bsresright y}{a \bsresleft y}, \\
  \pair{a}{x} / \pair{b}{y} & \assign \pair{a / b \wedge x \sresleft y}{x \sresright b}.
\end{align*}
  This was already observed in~\cite{tsinakis+wille06}. 

  Thirdly, in order to identify copies of $\alg{S}$ and $\alg{M}$ inside $\algSmodM$, we require that $\algM$ be equipped with a constant $\0 \in \algM$ which is \emph{cyclic} in the sense that
\begin{align*}
  a \ast \0 = \0 \ast a \text{ for each } a \in \algS.
\end{align*}
  Such a bimodule will be called a \emph{cyclic pointed} bimodule. Using the constant $\0$, we define the idempotent maps $\sigma$ and $\gamma$ on $\algSmodMzero$ as follows:
\begin{align*}
  & \sigma\colon \pair{a}{x} \mapsto  \pair{a}{a \ast \0} = \pair{a}{\0 \ast a}, \\
  & \gamma\colon \pair{a}{x} \mapsto \pair{\makebox[\auxlength]{$\0 \bsresright x$}}{x} = \pair{\makebox[\auxlength]{$x \sresleft \0$}}{x}.
\end{align*}
  Both are isotone maps, and $\sigma$ is moreover a semigroup homomorphism. Its image $(\algSmodM)_{\sigma}$ is therefore a subposemigroup of $\algSmodM$, while the image $(\algSmodM)_{\gamma}$ of the map $\gamma$ is a subposet of $\algSmodM$ which is a join semilattice. The posemigroup $(\algSmodM)_{\sigma}$ and the join semilattice $(\algSmodM)_{\gamma}$ are isomorphic to $\algS$ and $\algM$ via
\begin{align*}
  \makebox[\auxlengthtwo]{$\embedS$}\colon & a \mapsto \pair{a}{a \ast \0} = \pair{a}{\0 \ast a}, \\
 \embedM\colon & x \mapsto \pair{\makebox[\auxlength]{$\0 \bsresright x$}}{x} = \pair{\makebox[\auxlength]{$x \sresleft \0$}}{x}.
\end{align*}
  The action of the original $\alg{S}$-bimodule $\alg{M}$ is recovered, up to the isomorphism $\pair{\embedS}{\embedM}$, as the following action of $(\algSmodM)_{\sigma}$ on $(\algSmodM)_{\gamma}$: 
\begin{align*}
  a \ast x & \assign \gamma(a \cdot x), & x \ast a & \assign \gamma(x \cdot a).
\end{align*}
  To fully recover the original cyclic pointed bimodule from its restricted Nagata \mbox{product}, it remains to expand $\algSmodMzero$ by the constant $\embedM(\0)$.

  The only troublesome aspect of the above construction is that the resulting posemigroup might lack a multiplicative unit, even if $\alg{S}$ is a monoid and $\alg{M}$ is bounded. In order to obtain a unital structure, we will need to restrict to a subset of the Nagata product. While the Nagata product of a bimodule is an ordered algebra over the set of all pairs of the form $\pair{a}{x}$, the universe of the \emph{restricted} Nagata product of a cyclic pointed $\alg{S}$-bimodule $\alg{M}$ we restrict to the subset
\begin{align*}
  \algSmodMzero \assign \set{\pair{a}{x} \in \algSmodM}{a \ast \0 \leq x}.
\end{align*}
  This is a subposemigroup of $\algSmodM$, which is a residuated subposemigroup if the action of $\alg{S}$ on $\alg{M}$ is residuated and $\alg{S}$ is a residuated meet semilattice. If $\algS$ moreover has a multiplicative unit~$\1$ such that $\1 \ast x = x = x \ast \1$, then $\algSmodMzero$ is a monoid with multiplicative unit $\pair{\1}{\0}$. In fact, in this case $\algSmodMzero$ is obtained from $\algSmodM$ as the image of the \emph{double-division conucleus} $\delta_{p}$ of Galatos \& Jipsen~\cite{galatos+jipsen20} with respect to the element $p \assign \pair{\1}{\0}$. Restricting to $\algSmodMzero$ has another beneficial side effect, namely that the restriction of $\sigma$ to $\algSmodMzero$ is an interior operator (in addition to being a homomorphism of monoids) and the restriction of $\gamma$ to $\algSmodMzero$ is a closure operator. In particular, the map $\sigma$ is a \emph{conucleus} on $\algSmodMzero$, i.e.\ it satisfies
\begin{align*}
  \sigma m \circ \sigma n \leq \sigma(m \circ n).
\end{align*}

  The key property of the maps $\sigma$ and $\gamma$ which enables us to recover the bimodule action is the following compatibility condition, which we call \emph{structurality}:
\begin{align*}
  \sigma m \cdot \gamma n \leq \gamma (m \cdot n), \\
  \gamma n \cdot \sigma m \leq \gamma (n \cdot m).
\end{align*}
  Indeed, in any posemigroup $\algN$ equipped with a conucleus $\sigma$ and a closure operator $\gamma$ linked by this condition, the posemigroup of $\sigma$-open elements $\algN_{\sigma}$ acts on the poset of $\gamma$-closed elements $\algN_{\gamma}$ via the action defined above. Moreover, if the posemigroup $\algN$ is residuated, then the action of $\algNsigma$ on $\algNgamma$ is also residuated:
\begin{align*}
  x \bsresright y & = \sigma(x \bs y), & a \bsresleft x & = \gamma (a \bs x), \\
  x \sresleft y & = \sigma(x / y), & x \sresright a & = \gamma (x / a).
\end{align*}
  This yields what we call the \emph{structural bimodule} induced by $\langle \algN, \sigma, \gamma \rangle$. If $\algN$ is further equipped with a constant $\0 \in \algNgamma$, what we obtain is a pointed bimodule.

  These two constructions allow us to move back and forth between the two-sorted and single-sorted perspectives: the \emph{(restricted) Nagata \mbox{product} functor} trans\-forms a cyclic pointed residuated bimodule into a posemigroup equipped with ad\-ditional structure, and the \emph{structural bimodule functor} transforms a posemigroup with such additional structure into a pointed residuated bimodule.

  The main result of the present paper states that these two functors form an adjunction if we restrict to the appropriate category of single-sorted structures. This is encapsulated in the definition of Nagata posemigroups (Definition~\ref{def: nagata structures}) and Nagata residuated lattices (Definition~\ref{def: nagata rls}). The former is a minimalist definition: Nagata posemigroups have the minimal amount of structure for the adjunction to go through. The latter is maximalist: Nagata residuated lattices have all the meets, joins, and residuals one might wish for.

\begin{theoremA}[see Theorem~\ref{thm: nagata structures}]
  The restricted Nagata product functor from the \mbox{category} of cyclic pointed residuated bimodules to the \mbox{category} of Nagata posemigroups is right adjoint to the structural bimodule functor. The unit of the adjunction is the map ${m \mapsto \pair{\sigma m}{\gamma m}}$, the counit is the inverse of the iso\-morphism $\pair{\embedS}{\embedM}$.
\end{theoremA}

\begin{theoremB}[see Theorem~\ref{thm: nagata rls}]
  The restricted Nagata product functor from the \mbox{category} of cyclic pointed residuated lattice-ordered bimodules over a residuated lattice to the \mbox{category} of Nagata residuated lattices is the right adjoint of the \mbox{structural} bimodule functor. The unit of this adjunction is the embedding ${{m \mapsto \pair{\sigma m}{\gamma m}}}$, the counit is the inverse of the isomorphism $\pair{\embedS}{\embedM}$.
\end{theoremB}

  Nagata products, or more precisely twist products, are frequently studied in the signature with a \emph{bilattice} structure~\cite{rivieccio+maia+jung20}. Accordingly, we prove analogues of the above results for bilattices and sesquilattices (Definition~\ref{def: nagata rsls}) in Theorem~\ref{thm: nagata rls bilat}. In that case, we obtain a categorical equivalence, rather than a mere adjunction. Algebras of the latter kind feature only an additional semilattice structure on top of the usual lattice structure, rather than two lattice structures. This is necessitated by the fact that we talk about restricted Nagata products rather than the full Nagata product.

  A particular case of interest where bimodules arise is, as we saw above, when a residuated lattice acts on its order dual by division. Nagata products arising from bimodules of this form are the twist products of~\cite{tsinakis+wille06}, while the restricted Nagata products arising from such bimodules are the more general twist products of Busaniche et al.~\cite{busaniche+galatos+marcos22}. Twist products of this last kind in particular sub\-sume the twist product constructions of Sendlewski~\cite{sendlewski90} and Odintsov~\cite{odintsov04}.

  A more general set up, inspired by the non-involutive twist product of a linked pair of Heyting algebras studied by Rivieccio \& Spinks~\cite{rivieccio+spinks19}, arises from a pair of residuated lattices $\alg{L}_{+}$ and $\alg{L}_{-}$ linked by maps $\lambda\colon \alg{L}_{+} \to \alg{L}_{-}$ and $\rho\colon \alg{L}_{-} \to \alg{L}_{+}$. In this case, under suitable assumptions, $\alg{L}_{+}$ acts on $\alg{L}_{-}$ as follows:
\begin{align*}
  a \ast x & \assign x / \lambda a, & x \ast a & \assign \lambda a \bs x.
\end{align*}
  The Nagata product can then be equipped with the further operation
\begin{align*}
  \nagneg \pair{a}{x} & \assign \pair{\rho x}{\lambda a}.
\end{align*}
  We call quadruples $\langle \alg{L}_{+}, \alg{L}_{-}, \lambda, \rho \rangle$ of this kind \emph{twistable pairs} of residuated lattices (Definition~\ref{def: twist pair}), provided that the maps $\lambda$ and $\rho$ satisfy suitable conditions, and the Nagata construction applied to their associated modules is called the \emph{twist product} construction (Definition~\ref{def: twist products}). Theorems~\ref{thm: nagata posemigroup with negation adjunction}, \ref{thm: twist 2}, and \ref{thm: nagata structures with negation} are the analogues of the main adjunctions concerning Nagata products in the context of twist products. The twist product construction covers as special cases the non-involutive construction of quasi-Nelson algebras due to Rivieccio \& Spinks as well as the involutive construction of Busaniche et al. (See also~\cite{rivieccio+maia+jung20} for a detailed discussion of non-involutive twist products.) The latter construction deals with the case where the residuated lattices $\alg{L}_{+}$ and $\alg{L}_{-}$ essentially collapse into one, i.e.\ the case where $\lambda$ and $\rho$ are mutually inverse isomorphisms.

  In the final section, we use the Nagata construction to provide an alternative construction of the algebra of fractions of a Boolean-pointed Brouwerian algebra introduced in~\cite{galatos+prenosil22}. We show that this algebra of fractions is in fact the nucleus image of a conucleus image of a restricted Nagata product. This lends further credence to the intepretation of the restricted Nagata product as an algebra of formal fractions.

  One topic that we do not touch on in the present paper, despite it being a recurrent theme in the study of twist products, is the problem of restricting the Nagata adjunction to a categorical equivalence and the related problem of describing the left adjoint of the structural bimodule functor. The adjunctions of Sendlewski and Odintsov and in some cases also the adjunction of Busaniche et al.\ in fact yield categorical equivalences if we add further structure to the residuated lattices in question (in the form of a certain designated subset). This allows one to place further restrictions on which pairs $\pair{a}{x}$ are admitted in the twist product, which turns the units of these adjunctions into surjective maps. To obtain a Nagata equivalence rather than a Nagata adjunction would presumably require adding further structure to bimodules and moreover restricting to some suitable subclass of bimodules.


\setlength{\abovedisplayskip}{\abovedisplayskipaux}
\setlength{\belowdisplayskip}{\belowdisplayskipaux}

\section{Nuclei and conuclei on posemigroups}

  We start with a review of our terminology. The reader familiar with residuated lattices and related structures is advised to skip this review and refer to it only if if they need to. The only point which needs to emphasized to such readers is that we define conuclei as closure operators on posemigroups, therefore preserving the multiplicative unit is not part of the definition of a conucleus. A conucleus on a pomonoid which preserves the monoidal unit will be called a \emph{unital} conucleus.

  The most general algebraic structure that we shall encounter in this paper is a \emph{partially ordered semigroup} or \emph{posemigroup}. This is a structure $\algS = \langle S, \leq, \cdot \rangle$ such that $\langle S, \leq \rangle$ is a poset and $\langle S, \cdot \rangle$ is a semigroup whose multiplication is isotone in both arguments with respect to the partial order. A \emph{partially ordered monoid} or \emph{pomononoid} is both a posemigroup and a monoid. A \emph{pointed} posemigroup is one expanded by an (arbitrary) constant $\0$.

  A \emph{residuated posemigroup} is a structure $\algS = {\langle S, \leq, \cdot, \bs, / \rangle}$ such that ${\langle S, \leq, \cdot, \rangle}$ is a posemigroup and the binary operations $x \bs y$ and $x / y$ are the residuals of multiplication, i.e.\ they satisfy the equivalences
\begin{align*}
  y \leq x \bs z \iff x \cdot y \leq z \iff x \leq z / y.
\end{align*}
  A \emph{residuated pomonoid} is a structure $\algS = \langle S, \leq, \cdot, 1, \bs, / \rangle$ such that $\langle S, \leq, \cdot, \bs, / \rangle$ is a residuated posemigroup and $\langle S, \cdot, 1 \rangle$ is a monoid. A \emph{residuated subposemigroup (submonoid)} of a residuated posemigroup (pomonoid) is obtained by restricting to a subset closed with respect to both multiplication and residuation.

  A \emph{residuated $\ell$-semigroup} is a structure $\algS = \langle S, \wedge, \vee, \cdot, \bs, / \rangle$ such that $\langle S, \wedge, \vee \rangle$ is a lattice and $\langle S, \leq, \cdot, \bs, / \rangle$ is a residuated posemigroup ordered by the lattice order. A \emph{residuated lattice} is a structure $\algS = \langle S, \wedge, \vee, \cdot, 1, \bs, / \rangle$ such that $\langle S, \wedge, \vee, \cdot, \bs, / \rangle$ is a residuated $\ell$-semigroup and $\langle S, \cdot, 1 \rangle$ is a monoid. In other words, residuated lattices are unital residuated $\ell$-semigroups.

  An \emph{interior operator} on a poset $P$ is an isotone map $\sigma\colon P \to P$ and a \emph{closure operator} on a poset $P$ is an isotone map $\gamma\colon P \to P$ such that
\begin{align*}
  & \sigma \sigma x = \sigma x \leq x, & & x \leq \gamma x = \gamma \gamma x.
\end{align*}
  The posets of fixpoints of $\sigma$ and $\gamma$ will be denoted $P_{\sigma}$ and $P_{\gamma}$. If the join $x \vee y$ exists in $P$ for $x, y \in P_{\sigma}$, then the join also exists in $P_{\sigma}$ and coincides with $x \vee y$. If the meet $x \wedge y$ exists in $P$ for $x, y \in P_{\sigma}$, then the meet also exists in $P_{\sigma}$ and is computed as $\sigma(x \wedge y)$. If the join $x \vee y$ exists in $P$ for $x, y \in P_{\gamma}$, then the join also exists in $P_{\gamma}$ and is computed as $\gamma(x \vee y)$. If the meet $x \wedge y$ exists in $P$ for $x, y \in P_{\gamma}$, then the meet also exists in $P_{\gamma}$ and coincides with $x \wedge y$.

  A \emph{conucleus} on a posemigroup $\algS$ is an interior operator~$\sigma$ on $\algS$ such that
\begin{align*}
  \sigma x \cdot \sigma y & \leq \sigma (x \cdot y).
\end{align*}
  The image $\algS_{\sigma}$ of this interior operator $\sigma$ is a subposemigroup of $\algS$. If $\algS$ is residuated, then so is $\algS_{\sigma}$, the residuals being $\sigma(x \bs y)$ and $\sigma(x / y)$. If $\algS$ has a multiplicative unit~$\1$, then $\algS_{\sigma}$ has a multiplicative unit $\sigma(\1)$. We call the conucleus \emph{unital} if $\sigma(\1) = \1$. 

  A \emph{nucleus} on a posemigroup $\algS$ is a closure operator~$\gamma$ on $\algS$ such that
\begin{align*}
  \gamma x \cdot \gamma y & \leq \gamma (x \cdot y).
\end{align*}
  The image $\algS_{\gamma}$ of this closure operator $\gamma$ is a posemigroup with multiplication $x \cdotgamma y \assign \gamma(x \cdot y)$, the order being the restriction of the order of $\algS$. If $\algS$ is residuated, then so is $\algS_{\gamma}$, and the residuals in the two posemigroups coincide. If $\algS$ has a multiplicative unit~$\1$, then $\algS_{\gamma}$ has a multiplicative unit $\gamma(\1)$.

\section{Modules over posemigroups}

  We now introduce bimodules over posemigroups. Bimodules as defined here always consist (at least) of a posemigroup acting on a \emph{join semilattice}. One might also call such structures bisemimodules because they act on posemigroups rather than pogroups, but let us not complicate our terminology needlessly.

  Let $\algS = \langle S, \leq, \cdot \rangle$ be a posemigroup and $\algM = \langle M, \vee \rangle$ be a join semilattice. The semilattice order of $\algM$ will also be denoted by $\leq$. Throughout the paper, we adopt the following notational convention introduced in~\cite{tsinakis+wille06} whenever applicable:
\begin{align*}
  \textbf{$a, b, c$ are elements of $\algS$ (scalars) and $x, y, z$ are elements of~$\algM$.}
\end{align*}

  A \emph{biaction} on $\algS$ on $\algM$ (or more generally on a poset) consists of a \emph{left action} ${\ast\colon S \times M \to M}$ and a \emph{right action} ${\ast\colon M \times S \to M}$ which are isotone maps such that for all $x \in \algM$ and $a, b \in \algS$
\begin{align*}
  (a \cdot b) \ast x & = a \ast (b \ast x), \\
  x \ast (a \cdot b) & = (x \ast a) \ast b, \\
  (a \ast x) \ast b & = a \ast (x \ast b).
\end{align*}
  An element $\1 \in \algS$ is a \emph{unit} for the biaction if
\begin{align*}
  \1 \ast x = x = x \ast \1 \text{ for each } x \in \algM.
\end{align*}
  An element $\0 \in \algM$ is a \emph{zero} for the biaction if
\begin{align*}
  a \ast \0 = \0 = \0 \ast a \text{ for each } a \in \algS.
\end{align*}
  A biaction is \emph{commutative} if $a \ast x = x \ast a$ for each $a \in \algS$ and $x \in \algM$. Observe that we can afford to use the same symbol for the left and right action thanks to our notational convention about the types of $a, b, c$ and $x, y, z$.

  A \emph{bimodule}, or more explicitly an \emph{$\alg{S}$-bimodule}, is a two-sorted ordered algebra which consists of a posemigroup $\algS$, a join semilattice $\algM = \langle M, \vee \rangle$, and a biaction of $\algS$ on $\algM$ such that for all $x, y \in \algM$ and $a \in \algS$
\begin{align*}
  a \ast (x \vee y) & = (a \ast x) \vee (a \ast y), \\
  (x \vee y) \ast a & = (x \ast a) \vee (y \ast a).
\end{align*}

  A biaction of $\algS$ on $\algM$ (or more generally on a poset) is \emph{residuated} if there are maps ${\bsresleft\colon S \times M \to M}$ and $\sresleft\colon M \times M \to S$ (residuals of the left action) and maps ${\bsresright\colon M \times M \to S}$ and $\sresright\colon M \times S \to M$ (residuals of the right action) such that
\begin{align*}
  x \leq a \bsresleft y \iff a \ast x \leq y \iff a \leq y \sresleft x, \\
  x \leq y \sresright a \iff x \ast a \leq y \iff a \leq x \bsresright y.
\end{align*}
  The position of the star indicates whether the residual comes from the left action or the right action. A \emph{residuated bimodule} consists of a posemigroup $\alg{S}$, a join semilattice $\algM$, and a residuated biaction of $\alg{S}$ on $\alg{M}$ (together with the four residuals). The residuals of the biaction ensure that this is indeed a bimodule.

 A \emph{doubly residuated} bimodule is a residuated bimodule where moreover $\algS$ is a residuated posemigroup. That is, both the module action and scalar multiplication are residuated in a doubly residuated bimodule.

  A \emph{unital} bimodule is a bimodule where $\algS$ is equipped with a multiplicative unit $\1$ which is also a unit for the biaction. A \emph{pointed} bimodules is one endowed with a constant $\0 \in \algM$. In a \emph{cyclic} pointed bimodule $\0$ is a cyclic element, i.e.\ $a \ast \0 = \0 \ast a$ for each $a \in \algS$. An \emph{$\ell$-bimodule} is one where $\algS$ and $\algM$ are lattices.

  If $\algM$ is a residuated bimodule and the join $\bigvee_{i \in I} x_{i}$ exists in $\algM$, then
\begin{align*}
  a \ast \bigg( \bigvee_{i \in I} x_{i} \bigg) & = \bigvee_{i \in I} (a \ast x_{i}), & \bigg( \bigvee_{i \in I} x_{i} \bigg) \ast a & = \bigvee_{i \in I} (x_{i} \ast a).
\end{align*}
  In particular, if the semilattice $\algM$ has a bottom element $\botbold$, then
\begin{align*}
  \botbold \ast a = \botbold = a \ast \botbold.
\end{align*}
  Moreover, if $\algM$ has either a top element $\topbold$ or a bottom element~$\botbold$, then $\algS$ must have a top element. This is because for all $x \in \algM$ and $a \in \algS$
\begin{align*}
  x \ast a \leq \topbold & \iff a \leq x \bsresright \topbold, & a \ast x \leq \topbold & \iff a \leq \topbold \sresleft x, \\
  \botbold \ast a \leq \botbold & \iff a \leq \botbold \bsresright \botbold, & a \ast \botbold \leq \botbold & \iff a \leq \botbold \sresleft \botbold,
\end{align*}
  therefore $\botbold \bsresright \botbold = \botbold \sresleft \botbold = x \bsresright \topbold = \topbold \sresleft x$ (for any $x \in \alg{M}$) is the top element of $\algS$.

\section{The Nagata product}

  Let $\algS$ be a join semilattice and $\algM$ be an $\algS$-bimodule.

\begin{definition}[The Nagata product $\algSmodM$]
  The posemigroup $\algSmodM$ consists of the universe $S \times M$ with the componentwise order and the semigroup operation
\begin{align*}
  \pair{a}{x} \circ \pair{b}{y} & \assign \pair{a \cdot b}{x \ast b \vee a \ast y}.
\end{align*}
\end{definition}

  The Nagata product is a commutative posemigroup whenever $\alg{S}$ is a commutative posemigroups and the biaction is commutative. If $\algS$ and $\algM$ are both join (meet) semilattices, then so is $\algSmodM$:
\begin{align*}
  \pair{a}{x} \wedge \pair{b}{y} & \assign \pair{a \wedge b}{x \wedge y}, &
  \pair{a}{x} \vee \pair{b}{y} & \assign \pair{a \vee b}{x \vee y}.
\end{align*}
  The Nagata product of a doubly residuated bimodule is a residuated posemigroup, provided that $\algS$ is a meet semilattice:
\begin{align*}
  \pair{a}{x} \bs \pair{b}{y} & \assign \pair{a \bs b \wedge x \bsresright y}{a \bsresleft y}, &
  \pair{a}{x} / \pair{b}{y} & \assign \pair{a / b \wedge x \sresleft y}{x \sresright b}.
\end{align*}
  The requirement that $\alg{M}$ be a join semilattice plays a key role here:
\begin{align*}
  \pair{a}{x} \circ \pair{b}{y} \leq \pair{c}{z} & \iff a \cdot b \leq c \text{ and } x \ast b \vee a \ast y \leq z \\
  & \iff a \cdot b \leq c \text{ and } x \ast b \leq z \text{ and } a \ast y \leq z \\
  & \iff b \leq a \bs c \text{ and } b \leq x \bsresright z \text{ and } y \leq a \bsresleft z \\
  & \iff \pair{b}{y} \leq \pair{a \bs c \wedge x \bsresright z}{a \bsresleft z}. \qedhere
\end{align*}
  Putting the above observations together, we obtain Theorem~3.3 of~\cite{tsinakis+wille06}, which we restate below in our current terminology.

\begin{theorem}[Residuated $\ell$-semigroups arising from bimodules]
  The Nagata product of a doubly residuated $\ell$-bimodule is a residuated $\ell$-semigroup with the lattice structure
\begin{align*}
  \pair{a}{x} \wedge \pair{b}{y} & \assign \pair{a \wedge b}{x \wedge y}, \\
  \pair{a}{x} \vee \pair{b}{y} & \assign \pair{a \vee b}{x \vee y},
\end{align*}
  and the semigroup structure
\begin{align*}
  \pair{a}{x} \circ \pair{b}{y} & \assign \pair{a \cdot b}{x \ast b \vee a \ast y},
\end{align*}
  and the two residuals
\begin{align*}
  \pair{a}{x} \bs \pair{b}{y} & \assign \pair{a \bs b \wedge x \bsresright y}{a \bsresleft y}, \\
  \pair{a}{x} / \pair{b}{y} & \assign \pair{a / b \wedge x \sresleft y}{x \sresright b}.
\end{align*}
\end{theorem}

  One troublesome feature of the Nagata product construction is that it does not preserve unitality in a straight\-forward way. The remedy chosen by Tsinakis \& Wille was to assume that $\algM$ has a bottom element $\botbold$ which is a zero for the biaction, i.e.\ $a \ast \botbold = \botbold = \botbold \ast a$ for each $a \in \algS$. In that case $\algSmodM$ has a unit $\pair{\1}{\botbold}$.

  This is not entirely satisfactory. We shall see that \emph{any} element $\0 \in \algM$ will yield a multiplicative unit provided that we are willing to restrict to a suitable subalgebra of the Nagata product, which we call the \emph{restricted Nagata product}.

\begin{definition}[The restricted Nagata product $\algSmodMzero$]
  The posemigroup $\algSmodMzero$ is the subposemigroup of $\algSmodM$ with the universe
\begin{align*}
  \set{\pair{a}{x} \in \algSmodM}{\0 \ast a \leq x \text{ and } a \ast \0 \leq x}.
\end{align*}
\end{definition}  

  Observe that if $\0$ is the bottom element of $\algM$, then $a \ast \0 = \0 = \0 \ast a$ for all $a \in \algS$, and therefore $\algSmodMzero = \algSmodM$ in this case.

\begin{fact}
  The restricted Nagata product $\algSmodMzero$ of a pointed bimodule is indeed a subposemigroup of the Nagata product $\algSmodM$. Moreover:
\begin{enumerate}[(i)]
\item $\algSmodMzero$ has a multiplicative unit $\pair{1}{0}$ if the bimodule is unital.
\item $\algSmodMzero$ is a meet subsemilattice if $\algS$ and $\algM$ are meet semilattices.
\item $\algSmodMzero$ is a join subsemilattice if $\algS$ and $\algM$ are join semilattices and
\begin{align*}
  (a \vee b) \ast x & = (a \ast x) \vee (b \ast x), & x \ast (a \vee b) & = (x \ast a) \vee (x \ast b).
\end{align*}
\item $\algSmodMzero$ is a residuated subposemigroup if $\algS$ is a meet semilattice and the bimodule is doubly residuated.
\end{enumerate}
\end{fact}

\begin{proof}
  Throughout this proof, consider $\pair{a}{x}, \pair{b}{y} \in \algSmodMzero$. That is,
\begin{align*}
  a \ast \0 & \leq x, & \0 \ast b & \leq y, \\
  \0 \ast a & \leq x, & b \ast \0 & \leq y.
\end{align*}
  Then $\pair{a}{x} \circ \pair{b}{y} = \pair{a \cdot b}{x \ast b \vee a \ast y}$ in $\algSmodM$ and
\begin{align*}
  \0 \ast (a \cdot b) & = (\0 \ast a) \ast b \leq x \ast b \leq x \ast b \vee a \ast y, \\
  (a \cdot b) \ast \0 & = a \ast (b \ast \0) \leq a \ast y \leq x \ast b \vee a \ast y,
\end{align*}
  therefore $\pair{a}{x} \circ \pair{b}{y} \in\algSmodMzero$. Clearly $\pair{\1}{\0} \in \algSmodMzero$ if the bimodule is unital, since $\0 \ast \1 = \0 = \1 \ast \0$. The pair $\pair{\1}{\0}$ is a multiplicative unit of $\algSmodMzero$ because
\begin{align*}
  \pair{a}{x} \cdot \pair{\1}{\0} & = \pair{a \cdot \1}{a \ast \0 \vee x \ast \1} = \pair{a}{a \ast \0 \vee x} = \pair{a}{x}, \\
  \pair{\1}{\0} \cdot \pair{a}{x} & = \pair{\1 \cdot a}{\1 \ast x \vee \0 \ast a} = \pair{a}{x \vee \0 \ast a} = \pair{a}{x}. \qedhere
\end{align*}
  Moreover,
\begin{align*}
  (a \wedge b) \ast \0 & \leq a \ast \0 \leq x, & \0 \ast (a \wedge b) & \leq \0 \ast a \leq x, \\
  (a \wedge b) \ast \0 & \leq b \ast \0 \leq y, & \0 \ast (a \wedge b) & \leq \0 \ast b \leq y,
\end{align*}
  so $\pair{a \wedge b}{x \wedge y} \in \algSmodMzero$. Similarly,
\begin{align*}
  (a \vee b) \ast \0 & = (a \ast \0) \vee (b \ast \0) \leq x \vee y, & & \0 \ast (a \vee b) = (\0 \ast a) \vee (\0 \ast b) \leq x \vee y,
\end{align*}
  so $\pair{a \vee b}{x \vee y} \in \algSmodMzero$. Finally, $\pair{a}{x} \bs \pair{b}{y} = \pair{a \bs b \wedge x \bsresright y}{a \bsresright y}$ in $\algSmodM$ and
\begin{gather*}
  a \ast \0 \ast (a \bs b \wedge x \bsresright y) \leq x \ast (x \bsresright y) \leq y, \\
  a \ast ((a \bs b \wedge x \bsresright y) \ast \0) \leq (a \cdot (a \bs b)) \ast \0 \leq b \ast \0 \leq y,
\end{gather*}
  therefore $\0 \ast (a \bs b \wedge x \bsresright y) \leq a \bsresleft y$ and $(a \bs b \wedge x \bsresright y) \ast \0 \leq a \bsresleft y$, i.e.\ $\pair{a}{x} \bs \pair{b}{y} \in \algSmodM$. Similarly, we can show that $\pair{a}{x} / \pair{b}{y} \in \algSmodMzero$. 
\end{proof}

  Putting the above observations together yields the following theorem.

\begin{theorem}[Residuated lattices arising from pointed bimodules]
  The restricted Nagata product of a doubly residuated unital pointed $\ell$-bimodule is a residuated lattice $\algSmodMzero$ which is a residuated sub-$\ell$-semigroup of $\algSmodM$.
\end{theorem}

  The transition from $\algSmodM$ to $\algSmodMzero$, at least in the doubly residuated case, amounts to taking the image with respect to the \emph{double-division conucleus} of Galatos \& Jipsen~\cite{galatos+jipsen20}. Given a \emph{positive} element $p$ of a residuated posemigroup $\algN$, i.e.\ an element $p$ such that
\begin{align*}
  p \bs x, x / p \leq x \leq p \cdot x, x \cdot p \text{ for each } x \in \algN,
\end{align*}
  the double-division conucleus is the map
\begin{align*}
  \delta_{p} x \assign p \bs x / p.
\end{align*}
  The double-division conucleus provides a way of obtaining a residuated pomonoid from a residuated posemigroup: the image of this map is a residuated subposemigroup of $\algN$ which is a monoid with unit $p$. The image of $\delta_{p}$ consists of all elements of the form $p \bs x / p$, or equivalently of the form $p x p$, for some $x \in \algN$. Moreover, if $\algN$ is an $\ell$-semigroup, then the image of $\delta_{p}$ is a sub-$\ell$-semigroup of $\algN$.

  In our case, we take $p \assign \pair{\1}{\0}$. The image of $\delta_{p}$ then consists precisely of pairs $\pair{a}{x}$ such that $a \ast \0 \vee \0 \ast a \leq x$. In other words, $\algSmodMzero = (\algSmodM)_{\delta}$. This precisely mirrors the way double-division conuclei arise in the twist products studied by Busaniche et al.~\cite{busaniche+galatos+marcos22}, which we discuss in more detail later.\footnote{We are grateful to Nick Galatos for pointing out the relevance of the double-division conucleus in this context to us.}

\section{Recovering a pointed bimodule from its Nagata product}

  The (restricted) Nagata product can be constructed for any (pointed) bimodule. However, some restrictions are required if we are to recover the original (pointed) bimodule from its (restricted) Nagata product. Namely, we make the following assumption throughout this section:
\begin{align*}
  \textbf{$\algM$ is a cyclic pointed residuated $\algS$-bimodule.}
\end{align*}
  We shall define two idempotent maps $\sigma$ and $\gamma$ on $\algSmodM$ which will allow us to recover the $\algS$-bimodule $\algM$ by identifying $\algS$ with the image of $\sigma$ (via a map $\embedS$) and $\algM$ with the image of $\gamma$ (via a map $\embedM$).

  A meet $\bigwedge_{i \in i} \pair{a_{i}}{x_{i}}$ exists in $\algSmodM$ if and only if the meets $b \assign \bigwedge_{i \in I} a_{i}$ and $y \assign \bigwedge x_{i}$ exist in $\algS$ and $\algM$, in which case $\bigwedge_{i \in i} \pair{a_{i}}{x_{i}} = \pair{b}{y}$, and likewise for joins, since as a poset $\algSmodM$ is the product of the posets $\algS$ and $\algM$. If a meet which exists in $\algSmodM$ is computed in the same way, we call it a \emph{componentwise meet}. A non-trivial observation is that joins in $\algSmodMzero$ are always computed componentwise.
 
\begin{fact}
  The join $\bigvee_{i \in I} \pair{a_{i}}{x_{i}}$ exists in $\algSmodMzero$ if and only if the joins $a \assign \bigvee_{i \in I} a_{i}$ and $x \assign \bigvee_{i \in I} x_{i}$ exist in $\algM$ and $\algS$. In that case, $\bigvee_{i \in I} \pair{a_{i}}{x_{i}} = \pair{a}{x}$.
\end{fact}

\begin{proof}
  If the joins $a$ and $x$ exist, then $a \ast \0 = \bigvee_{i \in I} a_{i} \ast \0 \leq \bigvee_{i \in I} x_{i} \leq x$ because $\algM$ is residuated, hence $\pair{a}{x} \in \algSmodMzero$. It follows that $\pair{a}{x} = \bigvee_{i \in I} \pair{a_{i}}{x_{i}}$. Conversely, suppose that $\bigvee_{i \in I} \pair{a_{i}}{x_{i}} = \pair{b}{y}$ in $\algSmodMzero$. We show that $b = \bigvee_{i \in I} a_{i}$ and $y = \bigvee_{i \in I} x_{i}$. Suppose therefore that $a_{i} \leq c$ for each $i \in I$. Then $\pair{a_{i}}{x_{i}} \leq \pair{c}{c \ast \0 \vee y}$ for each $i \in I$, hence $\pair{b}{y} \leq \pair{c}{c \ast \0 \vee y}$ and $b \leq c$, showing that $\bigvee_{i \in I} a_{i} = b$. Similarly, suppose that $x_{i} \leq z$ for each $i \in I$. Then $\pair{a_{i}}{x_{i}} \leq \pair{\0 \bsresright z}{z}$ for each $i \in I$, since $\0 \ast a_{i} \leq x_{i} \leq z$, hence $\pair{b}{y} \leq \pair{\0 \bsresright z}{z}$ and $y \leq z$, showing that $\bigvee_{i \in I} x_{i} = y$. 
\end{proof}

\begin{proposition}[Recovering $\algS$ inside $\algSmodM$]
  The map $\sigma\colon \algSmodM \to \algSmodM$ defined as
\begin{align*}
  \sigma\colon \pair{a}{x} \mapsto \pair{a}{a \ast \0} = \pair{a}{\0 \ast a}
\end{align*}
  is an idempotent isotone map which preserves binary products and all existing joins. The restriction of $\sigma$ to $\algSmodMzero$ is a conucleus. Moreover, $(\algSmodMzero)_{\sigma} = (\algSmodM)_{\sigma}$. If $\algM$ is a unital $\algS$-bimodule, $\sigma$ is a unital conucleus on $\algSmodMzero$.
\end{proposition}

\begin{proof}
  The map $\sigma$ preserves binary products:
\begin{align*}
  \sigma \pair{a}{x} \circ \sigma \pair{b}{y} & = \pair{a}{a \ast \0} \circ \pair{b}{b \ast \0} \\ & = \pair{a \cdot b}{a \ast (b \ast \0) \vee (a \ast \0) \ast b} \\ & = \pair{a \cdot b}{(a \cdot b) \ast \0} \\ & = \sigma \pair{a \cdot b}{a \ast y \vee x \ast b} \\ & = \sigma \left( \pair{a}{x} \circ \pair{b}{y} \right).
\end{align*}
  The map also preserves all existing joins:
\begin{align*}
  \bigvee_{i \in I} \sigma \pair{a_{i}}{x_{i}} & = \bigvee_{i \in I} \pair{a_{i}}{a_{i} \ast \0} 
   = \pair{\bigvee_{i \in I} a_{i}}{\bigvee_{i \in I} (a_{i} \ast \0)} = \\
  & = \pair{\bigvee_{i \in I} a_{i}}{(\bigvee_{i \in I} a_{i}) \ast \0} 
   = \sigma \pair{\bigvee_{i \in I} a_{i}}{\bigvee_{i \in I} x_{i}} 
   = \sigma \bigvee_{i \in I} \pair{a_{i}}{x_{i}}.
\end{align*}
  The other claims are straightforward to observe.
\end{proof}

\begin{proposition}[Embedding $\algS$ into $\algSmodM$] \label{prop: bimodule embedding}
  The map $\embedS\colon \algS \to (\algSmodM)_{\sigma}$ defined as
\begin{align*}
  \embedS\colon a \mapsto \pair{a}{a \ast \0} = \pair{a}{\0 \ast a}
\end{align*}
  is an isomorphism. In particular, as a map $\embedS\colon \algS \to \algSmodM$ (or equivalently, as a map $\embedS\colon \algS \to \algSmodMzero$) it preserves products and all existing joins.
\end{proposition}

\begin{proof}
  The map $\embedS$ is an order embedding which preserves products:
\begin{align*}
  \embedS (a) \circ \embedS (b) & = \pair{a}{a \ast \0} \circ \pair{b}{b \ast \0} \\ & = \pair{a \cdot b}{a \ast (b \ast \0) \vee (a \ast \0) \ast b} \\ & = \pair{a \cdot b}{a \ast (b \ast \0) \vee a \ast (\0 \ast b)} \\ & = \pair{a \cdot b}{a \ast (b \ast \0)} \\ & = \pair{a \cdot b}{(a \cdot b) \ast 0} \\ & = \embedS (a \cdot b).
\end{align*}
  It is clearly surjective onto $(\algSmodM)_{\sigma}$, hence it is an isomorphism.
\end{proof}

\begin{proposition}[Recovering $\algM$ inside $\algSmodM$]
  The map $\gamma\colon \algSmodM \to \algSmodM$ defined as
\begin{align*}
  \gamma\colon \pair{a}{x} \mapsto \pair{\0 \bsresright x}{x} = \embedM(x) = \pair{x \sresleft \0}{x}
\end{align*}
  is an idempotent isotone map which preserves all existing meets. Its restriction to $\algSmodMzero$ is a closure operator which preserves all existing componentwise meets in $\algSmodMzero$. Moreover, $(\algSmodMzero)_{\gamma} = (\algSmodM)_{\gamma}$.
\end{proposition}

\begin{proof}
  The map also preserves all existing meets in $\algSmodM$ and therefore all componentwise meets in $\algSmodMzero$:
\begin{align*}
  \bigwedge_{i \in I} \gamma \pair{a_{i}}{x_{i}} & = \bigwedge_{i \in I} \pair{\0 \bsresright x_{i}}{x_{i}} 
   = \pair{\bigwedge_{i \in I} \0 \bsresright x_{i}}{\bigwedge_{i \in I} x_{i}} = \\
   & = \pair{\0 \bsresright \bigwedge_{i \in I} x_{i}}{\bigwedge_{i \in I} x_{i}} 
   = \gamma \pair{\bigwedge_{i \in I} a_{i}}{\bigwedge_{i \in I} x_{i}} 
   = \gamma \bigwedge_{i \in I} \pair{a_{i}}{x_{i}}. \qedhere
\end{align*}
  For $\pair{a}{x} \in \algSmodMzero$ we have $\pair{a}{x} \leq \embedM(y) = \pair{\0 \bsresright y}{y}$ if and only if $x \leq y$ and $a \leq \0 \bsresright y$. But $x \leq y$ implies $a \leq \0 \bsresright y$ ($\0 \ast a \leq x$ because $\pair{a}{x} \in \algSmodMzero$), therefore $\pair{a}{x} \leq \embedM(y)$ holds if and only if $x \leq y$. It follows that $\embedM(x)$ is the smallest pair of this form above $\pair{a}{x}$, therefore $\gamma$ is an interior operator on $\algSmodMzero$.
\end{proof}

\begin{proposition}[Embedding $\algM$ into $\algSmodM$]
  The map $\embedM\colon \algM \to (\algSmodM)_{\gamma}$ defined as
\begin{align*}
  \embedM\colon x \mapsto \pair{\0 \bsresright x}{x} = \pair{x \sresleft \0}{x}
\end{align*}
  is an isomorphism of posets. In particular, as a map $\embedM\colon \algM \to \algSmodMzero$ it preserves all existing meets in $\algSmodMzero$.
\end{proposition}

\begin{proof}
  The map is a surjective order embedding by definition, and the second claim holds because $\gamma$ is a closure operator on $\algSmodMzero$.
\end{proof}

  The map $\sigma$ thus assigns to each $m \in \algSmodMzero$ the largest element of the form $\embedS a$ below $m$, and $\gamma$ assigns to it the smallest element of the form $\embedM x$ above $m$.

\begin{proposition}[Residuals in $\algSmodM$]
  Residuals of the form $m \bs \gamma n$ and $\gamma n / m$ for $m, n \in \algSmodM$ exist in $\algSmodM$:
\begin{align*}
  \pair{a}{x} \bs \embedM y & = \pair{(\0 \ast a \vee x) \bsresright y}{a \bsresleft y}, & \embedM y / \pair{a}{x} = \pair{y \sresleft (x \vee a \ast \0)}{y \sresright a}.
\end{align*}
  Consequently, residuals of the form $m \bs \gamma n$ and $\gamma n / m$ for $m, n \in \algSmodMzero$ exist in $\algSmodMzero$. In $\algSmodMzero$ they simplify to:
\begin{align*}
  \pair{a}{x} \bs \embedM y & = \pair{x \bsresright y}{a \bsresleft y}, & \embedM y / \pair{a}{x} = \pair{y \sresleft x}{y \sresright a}.
\end{align*}
\end{proposition}

\begin{proof}
  Consider $\pair{a}{x}, \pair{b}{y} \in \algSmodM$ and $z \in \algM$:
\begin{align*}
  \pair{a}{x} \circ \pair{b}{y} \leq \embedM z & \iff \pair{a \cdot b}{x \ast b \vee a \ast y} \leq \pair{\0 \bsresright z}{z} \\
  & \iff a \cdot b \leq \0 \bsresright z \text{ and } x \ast b \leq z \text{ and } a \ast y \leq z \\
  & \iff \0 \ast (a \cdot b) \leq z \text{ and } x \ast b \leq z \text{ and } a \ast y \leq z \\
  & \iff (\0 \ast a) \ast b \leq z \text{ and } x \ast b \leq z \text{ and } a \ast y \leq z \\
  & \iff (\0 \ast a \vee x) \ast b \leq z \text{ and } a \ast y \leq z \\
  & \iff b \leq (\0 \ast a \vee x) \bsresright z \text{ and } y \leq a \bsresleft z,
\end{align*}
  where the penultimate equivalence holds because $x \ast b \leq z$ implies that $\0 \ast (a \cdot b) \leq (\0 \ast a) \ast b \leq x \ast b \leq z$.
\end{proof}

\begin{proposition}[Recovering the bimodule actions] \label{prop: recovering actions}
  For all $x, y \in \algS$ and $a \in \algM$ we have
\begin{align*}
  \embedM (x \ast a) & = \gamma(\embedM x \cdot \embedS a), & \embedS (x \bsresright y) & = \sigma(\embedM x \bs \embedM y), & \embedM (a \bsresleft x) & = \gamma (\embedS a \bs \embedM x), \\
  \embedM (a \ast x) & = \gamma(\embedS a \cdot \embedM x), & \embedS (x \sresleft y) & = \sigma(\embedM x / \embedM y), & \embedM (x \sresright a) & = \gamma (\embedM x / \embedS a).
\end{align*}
\end{proposition}

\begin{proof}
  We know that $\embedM z = \gamma (\embedM z)$ for all $z \in \algM$. Moreover, $\gamma \pair{b}{y} = \gamma \pair{c}{z}$ if and only if $y = z$, and $\sigma \pair{b}{y} = \sigma \pair{c}{z}$ if and only if $b = c$. The top three equations now hold because
\begin{align*}
  \embedM x \circ \embedS a & = \pair{x \sresleft \0}{x} \circ \pair{a}{\0 \ast a} \\
  & = \pair{(x \sresleft \0) \cdot a}{(x \sresleft \0) \ast \0 \ast a \vee x \ast a} \\
  & = \pair{(x \sresleft \0) \cdot a}{x \ast a}, \\ ~ \\
  \embedM x \bsresright \embedM y & = \pair{(\0 \bsresright x) \bs (\0 \bsresright y) \wedge x \bsresright y}{(\0 \bsresright x) \bsresleft y} \\
  & = \pair{x \bsresright y}{(\0 \bsresright x) \bsresleft y} \\ ~ \\
  \embedS a \bs \embedM x & =  \pair{a \bs (\0 \bsresright x) \wedge (a \ast \0) \bsresright x}{a \bsresleft x}.
\end{align*}
  The proofs of the bottom three equations are analogous.
\end{proof}

  We can therefore fully recover a cyclic pointed residuated $\algS$-bimodule $\algM$ from either the full Nagata product $\algSmodM$ or from the restricted Nagata product $\algSmodMzero$ using the maps $\sigma$ and $\gamma$ and the existence of residuals of the forms $m \bs \gamma n$ and $\gamma n / m$. 

\begin{theorem}[Recovering a bimodule from its Nagata product] \label{thm: counit is iso}
  Let $\algM$ be a cyclic pointed residuated $\algS$-bimodule and let $\algN \assign \algSmodM$ (or $\algN \assign \algSmodMzero$). Then the $\0$-pointed $\algS$-bimodule $\algM$ is isomorphic via the map $\pair{\embedS}{\embedM}$ to the $\embedM(\0)$-pointed $\algNsigma$-bimodule $\algNgamma$ with the action and the residuals of the action
\begin{align*}
  a \ast x & = \gamma(a \cdot x), & x \bsresright y & = \sigma(x \bs y), & a \bsresleft x & = \gamma(a \bs x), \\
  x \ast a & = \gamma(x \cdot a), & y \sresleft x & = \sigma(y / x), & x \sresright a & = \gamma(x / a).
\end{align*}
\end{theorem}

\section{The structural bimodule}
\label{sec: structural bimodule}

  In the previous section, we constructed two single-sorted algebras $\algSmodM$ and $\algSmodMzero$ out a cyclic pointed residuated $\algS$-bimodule~$\algM$. In this section, we consider the converse problem of constructing a bimodule given a posemigroup $\algN$ and two suitable idempotent isotone maps $\sigma$ and $\gamma$ on $\algN$. We then put the two constructions together and prove that they form an adjunction.

\begin{definition}[Conuclei and pre-conuclei]
  A \emph{(unital) pre-conucleus} on a posemigroup (pomonoid) $\algN$ is an idempotent isotone map $\sigma$ on $\algN$ such that $\algNsigma \assign \sigma[\algN]$ is a subsemigroup (submonoid) of $\algN$. A \emph{conucleus} is a pre-conucleus $\sigma$ such that $\sigma(m) \leq m$ for all $m \in \algN$.
\end{definition}

\begin{definition}[$\sigma$-structural $\sigma$-closure operators]
  Given a pre-conucleus $\sigma$ on a posemigroup $\algN$, a \emph{$\sigma$-closure operator} is an idempotent isotone map $\gamma$ on $\algN$ such that for all $a \in \algNsigma$ and $x \in \algNgamma \assign \gamma[\algN]$
\begin{align*}
  a \cdot x & \leq \gamma(a \cdot x), \\
  x \cdot a & \leq \gamma(x \cdot a).
\end{align*}
  A $\sigma$-closure operator is \emph{$\sigma$-structural} if for all $a \in \algNsigma$ and $m \in \algN$
\begin{align*}
  a \cdot \gamma m & \leq \gamma(a \cdot m), \\
  \gamma m \cdot a & \leq \gamma(m \cdot a).
\end{align*}
\end{definition}

  Structurality may be defined in several equivalent ways.

\begin{fact} \label{fact: gamma}
  Let $\sigma$ be a pre-conucleus on a posemigroup $\algN$ and let $\gamma$ be a $\sigma$-closure operator. Then the following conditions are equivalent:
\begin{enumerate}[(i)]
\item $a \cdot \gamma m \leq \gamma (a \cdot m)$ for each $a \in \algNsigma$ and $m \in \algN$.
\item $\sigma m \cdot \gamma n \leq \gamma (\sigma m \cdot n)$ for each $m, n \in \algN$.
\item $\gamma (m \bs n) \leq \sigma m \bs \gamma n$ for each $m, n \in \algN$.
\item $\gamma (a \bs m) \leq a \bs \gamma m$ for each $x \in \algN$ and $a \in \algNsigma$.
\end{enumerate}
  In case $\gamma$ is a closure operator, these conditions are equivalent to:
\begin{enumerate}[(i)]
\setcounter{enumi}{4}
\item if $a \in \algNsigma$ and $x \in \algNgamma$, then $a \bs x \in \alg{N}_{\gamma}$.
\end{enumerate}
\end{fact}

\begin{proof}
  The equivalences (i) $\Leftrightarrow$ (ii) and (iii) $\Leftrightarrow$ (iv) and the implications (i) $\Rightarrow$ (iii) and (ii) $\Rightarrow$ (iv) and (iv) $\Rightarrow$ (v) are immediate. To prove (iii) $\Rightarrow$ (i):
\begin{align*}
  \sigma a \cdot \gamma x \leq {\sigma a \cdot \gamma (a \bs ax)} \leq \sigma a \cdot \sigma a \bs \gamma (ax) \leq \gamma (ax).
\end{align*}
  Finally, to prove (v) $\Rightarrow$ (iv): $\gamma(a \bs x) \leq \gamma (a \bs \gamma x) \leq a \bs \gamma x$ if $a \in \algNsigma$.
\end{proof}

  As we saw in the previous section, while a (restricted) Nagata product need not be residuated, residuals of the forms $m \bs \gamma n$ and $\gamma n / m$ are guaranteed to exist. Similarly, while the (restricted) Nagata product need not be a join semilattice, the image of $\gamma$ is a subposet which is a join semilattice.\footnote{Restricting to the case where all residuals exist in the restricted Nagata product would force us to restrict to bimodules where $\alg{S}$ is a meet semilattice, since this meet operation is used in the definition of $\pair{a}{x} \bs \pair{b}{y}$ and $\pair{b}{y} / \pair{a}{x}$. This would in turn force us to further restrict to bimodules where $\alg{M}$ is a lattice, in order to ensure that $\algSmodMzero$ is a meet semilattice from which the meet semilattice structure of $\alg{S}$ can be recovered. Our motivation here is not so much to assume the existence of the least amount of residuals possible, but rather to avoid assuming all of this additional semilattice and lattice structure.}

\begin{definition}[$\gamma$-residuated posemigroups and $\gamma$-join semilattices]
  Let $\gamma$ be an idempotent isotone map on a posemigroup $\algN$. Then $\algN$ is \emph{$\gamma$-residuated} if for each $m \in \algN$ and $x \in \algNgamma$ the residuals $m \bs x$ and $x / m$ exist. We take a $\gamma$-residuated posemigroup to be an expansion of $\algN$ by the operations
\begin{align*}
  m \bsgamma n & \assign m \bs \gamma n, \\
  n \sgamma m & \assign \gamma n / m.
\end{align*}
  Moreover, $\algN$ is a \emph{$\gamma$-join semilattice} if for all $x, y \in \algNgamma$ their join $x \veegamma y$ in $\algNgamma$ exists. We take a $\gamma$-join semilattice to be an expansion of $\algN$ by the operation
\begin{align*}
  m \sqcup n \assign \gamma m \veegamma \gamma n.
\end{align*}
\end{definition}

  Although $m \bsgamma \gamma n$ is by definition equal to $m \bsgamma n$, and likewise $\gamma m \sqcup \gamma n$ is by definition equal to $m \sqcup n$, we shall prefer to write $m \bsgamma \gamma n$ and $\gamma m \sqcup \gamma n$. That is, we use the symbols $\bsgamma$ and $\sqcup$ merely to  emphasize that not all residuals or joins are required to exist, rather than as a notational abbreviation.

  The reason for introducing these operations is to ensure that the classes of algebras introduced below can be defined by inequalities or quasi-inequalities.

\begin{proposition}[The structural bimodule induced by $\gamma$ and $\sigma$] \label{prop: structural action}
  Let $\sigma$ be a (unital) pre-conucleus on a posemigroup (pomonoid) $\algN$ and $\gamma$ be a $\sigma$-structural $\sigma$-closure operator on $\algN$. Then $\algNsigma$ has a (unital) action on $\algNgamma$:
\begin{align*}
  a \ast x & \assign \gamma(a \cdot x), \\
  x \ast a & \assign \gamma(x \cdot a).
\end{align*}
  If $\algN$ is $\gamma$-residuated, then the action $\ast$ is residuated, the residuals being
\begin{align*}
  x \bsresright y & = \sigma(x \bsgamma y), & a \bsresleft x & = \gamma(a \bsgamma x), \\
  y \sresleft x & = \sigma(y \sgamma x), & x \sresright a & = \gamma(x \sgamma a).
\end{align*}
  If $\alg{N}_{\gamma}$ is moreover a join semilattice, then the above action gives $\alg{N}_{\gamma}$ the structure of a residuated $\alg{N}_{\sigma}$-bimodule. An element $\0 \in \algNgamma$ is cyclic in this bimodule if and only if $\gamma(a \cdot \0) = \gamma(\0 \cdot a)$ for all $a \in \algNsigma$.
\end{proposition}

\begin{proof}
  In the following, let $a, b \in \algNsigma$ and $x, y \in \algNgamma$. We only prove the claims for the left action. If $\algN$ has a unit $\1$ and the conucleus $\sigma$ is unital, then $\1 \ast x = \gamma (\1 \cdot x) = x$. Moreover,
\begin{align*}
  & ab \ast x = \gamma (ab \cdot x) = \gamma (a \cdot (b \cdot x)) \leq \gamma (a \cdot \gamma (b \cdot x)) = a \ast (b \ast x), \\
  & a \ast (b \ast x) = \gamma (a \cdot \gamma (b \cdot x)) \leq \gamma \gamma (a \cdot b \cdot x) = \gamma (ab \cdot x) = ab \ast x.
\end{align*}
  Likewise,
\begin{align*}
  & a \ast (x \ast b) = \gamma(a \cdot \gamma (x \cdot b)) \leq \gamma (a \cdot x \cdot b) \leq \gamma( \gamma(a \cdot x) \cdot b) = (a \ast x) \ast b, \\
  & (a \ast x) \ast b = \gamma(\gamma (a \cdot x) \cdot b) \leq \gamma (a \cdot x \cdot b) \leq \gamma (a \cdot \gamma (x \cdot b)) = a \ast (x \ast b),
\end{align*}
  therefore $\algNsigma$ acts (unitally) on $\algNgamma$. It remains to find the residuals of this action.

  If $x \leq \gamma(a \bsgamma y)$, then $\gamma(a \cdot x) \leq \gamma(a \cdot \gamma(a \bsgamma y)) \leq \gamma(a \cdot (a \bsgamma y)) \leq \gamma y = y$. Conversely, $a \cdot x \leq \gamma(a \cdot x)$ implies that $x \leq a \bsgamma \gamma(a \cdot x)$ and $x \leq \gamma(a \bsgamma \gamma(a \cdot x))$ by applying $\gamma$ to the previous inequality, so if $\gamma(a \cdot x) \leq y$, then $x \leq \gamma(a \bsgamma \gamma(a \cdot x)) \leq \gamma(a \bsgamma y)$.

  If $a \leq \sigma(y \sgamma x)$, then $\gamma(a \cdot x) \leq \gamma(\sigma(y \sgamma x) \cdot x) \leq \gamma((y \sgamma x) \cdot x) \leq \gamma y = y$. Conversely, $a \cdot x \leq \gamma(a \cdot x)$ implies that $a \leq \gamma(a \cdot x) \sgamma x$ and $a \leq \sigma(\gamma(a \cdot x) \sgamma x)$ by applying $\sigma$ to the previous inequality, so if $\gamma(a \cdot x) \leq y$, then $a \leq \sigma(\gamma(a \cdot x) \sgamma x) \leq \sigma(y \sgamma x)$.
\end{proof}

  The (residuated) bimodule produced by the above construction will be called the \emph{structural bimodule} induced by $\langle \algN, \sigma, \gamma \rangle$.

\begin{proposition}[The $\sigma$-structurality of $\gamma$]
  The map $\gamma$ is a $\sigma$-structural $\sigma$-closure operator on $\algSmodM$ (hence also on $\algSmodMzero$).
\end{proposition}

\begin{proof}
   Consider $a \in \algS$ and $x \in \algM$. Then
\begin{align*}
  \embedS(a) \circ \embedM(x) = \pair{a (\0 \bsresright x)}{a \ast x \vee a \ast 0 \ast \0 \bsresright x} = \pair{a (x \sresleft \0)}{a \ast x},
\end{align*}
  and $\pair{a (x \sresleft \0)}{a \ast x} \leq \pair{(a \ast x) \sresleft \0}{a \ast x} = \gamma (\pair{a (x \sresleft \0)}{a \ast x})$, since $a (x \sresleft \0) \leq (a \ast x) \sresleft \0$. Thus $\gamma$ is a $\sigma$-closure operator. Moreover,
\begin{align*}
  \embedS(a) \circ \gamma \pair{b}{y} \leq \gamma(\embedS(a) \circ \gamma \pair{b}{y}) = \pair{\0 \bsresright (a \ast y)}{a \ast y} = \embedM(a \ast y),
\end{align*}
  and $\embedM(a \ast y) \leq \embedM(a \ast y \vee (a \ast \0) \ast b) = \gamma(\embedS(a) \circ \pair{b}{y})$. Thus $\gamma$ is $\sigma$-structural.
\end{proof}

  Now that we have identified the conditions under which a pair of maps $\sigma$ and $\gamma$ on a posemigroup $\algN$ induces a bimodule, it only remains to identify the conditions under which $\algN$ embeds into the (restricted) Nagata product of this bimodule.

\begin{definition}[Nagata posemigroups] \label{def: nagata structures}
  A \emph{Nagata posemigroup} consists of a posemigroup $\algN$ equipped with a pre-conucleus~$\sigma$, a $\sigma$-structural $\sigma$-closure operator~$\gamma$, and a constant $\0 \in \algNgamma$ such that $\algN$ is a $\gamma$-residuated $\gamma$-join semilattice which satisfies the following equations:
\begin{align*}
  \sigma (x \cdot y) & = \sigma x \cdot \sigma y &
  \gamma (x \cdot y) & = \gamma(\gamma x \cdot \sigma y) \sqcup \gamma(\sigma x \cdot \gamma y)
\end{align*}
\begin{align*}
  \sigma(x \bsgamma \gamma y) & = \sigma(\gamma x \bsgamma \gamma y) & \gamma (x \bsgamma \gamma y) & = \sigma x \bsgamma \gamma y \\
  \sigma (\gamma y \sgamma x) & = \sigma(\gamma y \sgamma \gamma x) & \gamma (\gamma y \sgamma x) & = \gamma y \sgamma \sigma x
\end{align*}
\begin{align*}
  \sigma \gamma x & = \sigma(\0 \bsgamma \gamma x) & \gamma \sigma x & = \gamma(\0 \cdot \sigma x) \\
  \sigma \gamma x & = \sigma(\gamma x \sgamma \0) & \gamma \sigma x & = \gamma(\sigma x \cdot \0)
\end{align*}
  as well as the quasi-inequality
\begin{align*}
  \sigma x \leq \sigma y \text{ and } \gamma x \leq \gamma y \implies x \leq y.
\end{align*}
  A \emph{restricted Nagata posemigroup} is a Nagata posemigroup such that $\sigma$ is an interior operator (hence a conucleus) and $\gamma$ is a closure operator.
\end{definition}

  The inequalities
\begin{align*}
  & \gamma(x \bsgamma \gamma y) \leq \sigma x \bsgamma \gamma y, & & \gamma(\gamma y \sgamma x) \leq \gamma y \sgamma \sigma x,
\end{align*}
  follow from the $\sigma$-structurality of $\gamma$. The equalities
\begin{align*}
  \gamma (x \bsgamma \gamma y) & = \sigma x \bsgamma \gamma y, & \gamma (\gamma y \sgamma x) & = \gamma y \sgamma \sigma x
\end{align*}
  in the above axiomatization may equivalently be replaced by the seemingly weaker equalities
\begin{align*}
  \gamma (x \bsgamma \gamma y) & = \gamma(\sigma x \bsgamma \gamma y), & \gamma (\gamma y \sgamma x) & = \gamma(\gamma y \sgamma \sigma x),
\end{align*}
  since we shall only need this latter pair of equalities to prove Theorem~\ref{thm: nagata structures} (from which the seemingly stronger equalities follow).

  In the case of restricted Nagata posemigroups, the equalities
\begin{align*}
  \sigma(x \bsgamma \gamma y) & = \sigma(\gamma x \bsgamma \gamma y), & & \sigma (\gamma y \sgamma x) = \sigma(\gamma y \sgamma \gamma x)
\end{align*}
  are redundant, since they hold whenever $\sigma$ is an interior operator and $\gamma$ is a closure operator: the left-below-right inequalities follow immediately from $x \leq \gamma x$, and the right-below-left inequalities hold because
\begin{align*}
  \sigma(\gamma x \bsgamma \gamma y) \leq \sigma(x \bsgamma \gamma y) \iff \sigma(\gamma x \bsgamma \gamma y) \leq x \bsgamma \gamma y \iff x \cdot \sigma(\gamma x \bsgamma \gamma y) \leq \gamma y
\end{align*}
  and $x \cdot \sigma(\gamma x \bsgamma \gamma y) \leq \gamma x \cdot (\gamma x \bsgamma y) \leq \gamma y$.

\begin{fact} \label{fact: nagata posemigroups}
  The (restricted) Nagata product of a cyclic pointed residuated bimodule is a (restricted) Nagata posemigroup.
\end{fact}

\begin{proof}
  We already know that the (restricted) Nagata product is a posemigroup where $\sigma$ is an idempotent isotone map (a conucleus) with $\sigma(x \circ y) = \sigma x \circ \sigma y$, that $\gamma$ is a $\sigma$-structural $\sigma$-closure operator (closure operator), and that the required joins and residuals exist. Moreover, $\embedM(\0)$ is a fixpoint of~$\gamma$ by definition. The equalities displayed above are mechanical to verify, as they merely express that the first or the second components of certain expressions in a Nagata product coincide. For example, the equality $\gamma(x \circ y) = (\gamma x \circ \sigma y) \sqcup (\sigma x \circ \gamma y)$ holds because the second components of $(\gamma \pair{a}{x} \circ \sigma \pair{b}{y})$ and $(\sigma \pair{a}{x} \circ \gamma \pair{b}{y})$ are $a \ast y$ and $x \ast b$, while the second component of $\pair{a}{x} \circ \pair{b}{y}$ is $a \ast y \vee x \ast b$. Similarly, the equality $\sigma x \bsgamma \gamma y = \gamma (x \bsgamma \gamma y)$ holds because
\begin{align*}
  \pair{a}{a \ast \0} \bsgamma \pair{\0 \bsresright y}{y} = \pair{a \bsresright (\0 \bsresright y) \wedge (a \ast \0) \bsresright y}{a \bsresright y} = \pair{\0 \bsresright (a \bsresright y)}{a \bsresright y} = \embedM(a \bsresright y)
\end{align*}
  and $\gamma (\pair{a}{x} \bsgamma \pair{\0 \bsresright y}{y}) = \embedM(a \bsresright y)$. Finally, the equalities involving the constant $\0$ hold because the second components of $\pair{a}{a \ast \0}$, $\pair{\0 \bsresright \0}{\0} \cdot \pair{a}{a \ast \0}$, and $\pair{a}{a \ast \0} \cdot \pair{\0 \bsresright \0}{\0}$ coincide, as do the first components of $\pair{\0 \bsresright x}{x}$, $\pair{\0 \bsresright \0}{\0} \bs \pair{\0 \bsresright x}{x}$, and $\pair{\0 \bsresright x}{x} / \pair{\0 \bsresright \0}{\0}$.
\end{proof}                                                       

\begin{theorem}[Adjunction for Nagata posemigroups] \label{thm: nagata structures}
  The (restricted) Nagata product functor going from the \mbox{category} of cyclic pointed residuated bimodules to the \mbox{category} of (restricted) Nagata posemigroups is right adjoint to the structural bimodule functor. The unit of the adjunction is the embedding ${m \mapsto \pair{\sigma m}{\gamma m}}$, the counit is the inverse of the iso\-morphism $\pair{\embedS}{\embedM}$.
\end{theorem}

\begin{proof}
  We already know that the (restricted) Nagata product is a (restricted) Nagata posemigroup and that the structural bimodule of a (restricted) Nagata posemigroup is a cyclic pointed residuated bimodule. We also know that the counit is an isomorphism of cyclic pointed residuated bimodules (Theorem~\ref{thm: counit is iso}). The map $\unit$ is injective because $m \nleq n$ implies that either $\sigma m \nleq \sigma n$ or $\gamma m \nleq \gamma n$. The triangle equalities are easy to verify: one of them yields the sequence $\pair{a}{x} \mapsto \pair{\pair{a}{a \ast \0}}{\pair{\0 \bsresright x}{x}} \mapsto \pair{a}{x}$, while the other yields the two sequences $a \mapsto \pair{a}{\gamma(a \cdot \0)} \mapsto a$ and $x \mapsto \pair{\0 \bsresright x}{x} \mapsto x$. The fact that for each element $m$ of a restricted Nagata posemigroup the pair $\pair{\sigma m}{\gamma m}$ indeed lies in the restricted Nagata product of the structural bimodule, i.e.\ that $\sigma m \ast \0 \leq \gamma m$, holds because $\sigma m \ast \0 = \gamma(\sigma m \cdot \0) = \gamma \sigma m \leq \gamma m$, since $\sigma m \leq m$ in a restricted Nagata posemigroup. It remains to verify that $\unit$ is a homomorphism of Nagata posemigroups.

  The equations in the definition of Nagata posemigroups are precisely those needed for this purpose. The equations involving product are equivalent to $\unit$ preserving multiplication: $\unit(x \cdot y) = \pair{\sigma (x \cdot y)}{\gamma (x \cdot y)}$ and
\begin{align*}
  \unit(x) \circ \unit(y) & = \pair{\sigma x}{\gamma x} \circ \pair{\sigma y}{\gamma y} = \pair{\sigma x \cdot \sigma y}{\sigma x \ast \gamma y \sqcup \gamma x \ast \sigma y} = \\
  & = \pair{\sigma x \cdot \sigma y}{\gamma(\sigma x \cdot \gamma y) \sqcup \gamma(\gamma x \cdot \sigma y)}.
\end{align*}
  The equations involving residuals are equivalent to $\unit$ preserving the appropriate residuals: $\unit(x \bsgamma y) = \pair{\sigma (x \bsgamma y)}{\gamma (x \bsgamma y)}$ and
\begin{align*}
  \unit(x) \bsgamma \unit(y) = \pair{\sigma x}{\gamma x} \bsgamma \pair{\sigma y}{\gamma y} = \pair{\gamma x \bsresright \gamma y}{\sigma x \bsresleft \gamma y} = \pair{\sigma(\gamma x \bsgamma y)}{\gamma(\sigma x \bsgamma y)},
\end{align*}
  Finally, the equality between $\unit(\sigma x) = \pair{\sigma x}{\gamma \sigma x}$ and $\sigma \unit(x) = \pair{\sigma x}{\sigma x \ast \0} = \pair{\sigma x}{\gamma(\sigma x \cdot \0)}$, the equality between $\unit(\gamma x) = \pair{\sigma \gamma x}{\gamma x}$ and $\gamma \unit(x) = \pair{\0 \bsresright \gamma x}{\gamma x} = \pair{\sigma(\0 \bsgamma x)}{\gamma x}$, and the equality between $\unit(\0) = \pair{\sigma \0}{\gamma \0}$ and $\pair{\0 \bsresright \0}{\0} = \pair{\sigma(\0 \bsgamma \0)}{\0}$ (which is the interpretation of the constant $\0$ in the Nagata product) are immediate consequences of the last four equations defining Nagata posemigroups.
\end{proof}

  The above theorem is a minimalist, bare-bones form of the Nagata adjunction: it contains nothing more than what is strictly needed to make the adjunction work. Adding further structure to both sides of the adjunction is now simply a matter of verifying some further equations. In particular, we flesh out the maximalist variant of the Nagata adjunction below, where we assume as much lattice-theoretic and residuated structure as one can possibly want on both sides of the adjunction.

\begin{definition}[Nagata lattices] \label{def: nagata rls}
  A \emph{Nagata lattice} is a residuated $\ell$-semigroup $\algN$ equipped with a pre-conucleus~$\sigma$, a $\sigma$-structural $\sigma$-closure operator $\gamma$, and constants $\1 \in \algNsigma$ and $\0 \in \algNgamma$ such that the following equations and quasi-equations are satisfied:
\begin{align*}
  \sigma (x \cdot y) & = \sigma x \cdot \sigma y & \gamma(x \cdot y) & = \sigma x \cdot \gamma y \vee \gamma x \cdot \sigma y \\
  \sigma (x \wedge y) & = \sigma (\sigma x \wedge \sigma y) & \gamma (x \vee y) & = \gamma (\gamma x \vee \gamma y) \\
  \sigma (x \vee y) & = \sigma x \vee \sigma y & \gamma (x \wedge y) & = \gamma x \wedge \gamma y
\end{align*}
\begin{align*}
  \sigma(x \bs y) & = \sigma(\sigma x \bs \sigma y \wedge \gamma x \bs \gamma y) & \sigma x \bs \gamma y & = \gamma (x \bs y) \\
  \sigma(y / x) & = \sigma(\sigma y / \sigma x \wedge \gamma y / \gamma x) & \gamma y / \sigma x & = \gamma (y / x)
\end{align*}
\begin{align*}
  \sigma \gamma x & = \sigma(\0 \bs \gamma x) & \gamma \sigma x & = \gamma(\0 \cdot \sigma x) \\
  \sigma \gamma x & = \sigma(\gamma x / \0) & \gamma \sigma x & = \gamma(\sigma x \cdot \0)
\end{align*}
\begin{align*}
  & \1 \leq \sigma(x \bs x) & \1 \leq \sigma(x \bs y) & \implies x \leq y \\
  & \1 \leq \sigma(x / x) & \1 \leq \sigma(y / x) & \implies x \leq y
\end{align*}
  A \emph{restricted Nagata lattice} is a residuated lattice $\algN$ equipped with a unital conucleus $\sigma$, a $\sigma$-structural closure operator $\gamma$, and a constant $\0 \in \algNgamma$ such that the following equations are satisfied:
\begin{align*}
  \sigma(x \cdot y) & = \sigma x \cdot \sigma y & x \cdot y & = \sigma x \cdot y \vee x \cdot \sigma y
\end{align*}
\begin{align*}
  \sigma x \bs y \wedge x \bs \gamma y & = x \bs y & \sigma x \bs \gamma y & = \gamma (x \bs y) \\
  x / \sigma y \wedge \gamma x / y & = x / y & \gamma y / \sigma x & = \gamma (y / x)
\end{align*}
\begin{align*}
  \sigma \gamma x & = \sigma(\0 \bs \gamma x) & \gamma \sigma x & = \gamma(\0 \cdot \sigma x) \\
  \sigma \gamma x & = \sigma(\gamma x / \0) & \gamma \sigma x & = \gamma(\sigma x \cdot \0)
\end{align*}
\end{definition}

\begin{lemma} \label{lemma: nagata implication}
  Every (restricted) Nagata lattice satisfies the implication
\begin{align*}
  \sigma x \leq \sigma y \text{ and } \gamma x \leq \gamma y \implies x \leq y.
\end{align*}
\end{lemma}

\begin{proof}
  In the case of Nagata lattices, if $\sigma x \leq \sigma y$ and $\gamma x \leq \gamma y$, then $\1 \leq \sigma x \bs \sigma y \wedge \gamma x \bs \gamma y$, so $\1 = \sigma(\1) \leq \sigma(\sigma x \bs \sigma y \wedge \gamma \bs \gamma y) = \sigma(x \bs y)$ and $x \leq y$ by the quasi-equation $\1 \leq \sigma(x \bs y) \implies x \leq y$. In the case of restricted Nagata lattices, if $\sigma x \leq \sigma y$ and $\gamma x \leq \gamma y$, then $\sigma x \leq y$ and $x \leq \gamma y$, so $\1 \leq \sigma x \bs y \wedge x \bs \gamma y = x \bs y$ and $x \leq y$.
\end{proof}

\begin{fact}
  The (restricted) Nagata product of a cyclic pointed residuated $\ell$-bimodule over a residuated lattice is a (restricted) Nagata lattice.
\end{fact}

\begin{proof}
  The equalities $\sigma(x \vee y) = \sigma x \vee \sigma y$ and $\gamma(x \wedge y) = \gamma x \wedge \gamma y$ are straightforward to verify. The equalities involving $\0$ were already verified in Fact~\ref{fact: nagata posemigroups}. The equalities involving $\gamma (x \bs y)$ and $\gamma (y / x)$ plus the equalities $\sigma (x \cdot y)$ and $\gamma(x \cdot y)$ are verified in the same way as the corresponding equalities in Fact~\ref{fact: nagata posemigroups}. The equality $\1 \leq \sigma(x \bs x)$ holds because $\1 = \sigma(\1)$ and the left component of $\pair{a}{x} \bs \pair{a}{x}$ is $a \bs a \wedge x \bsresright x \geq \1$. If $\1 \leq \sigma(\pair{a}{x} \bs \pair{b}{y})$, then, looking at the left component, $\1 \leq a \bs b \wedge x \bsresright y$, so indeed $a \leq b$ and $x \leq y$.

  It only remains to verify the equalities $\sigma x \bs y \wedge x \bs \gamma y = x \bs y$, $x / \sigma y \wedge \gamma x / y = x / y$, and $x \circ y = \sigma x \circ y \vee x \circ \sigma y$ in the restricted case. We have:
\begin{align*}
  \sigma \pair{a}{x} \circ \pair{b}{y} \vee \pair{a}{x} \circ \sigma \pair{b}{y} & = \pair{a}{a \ast \0} \circ \pair{b}{y} \vee \pair{a}{x} \circ \pair{b}{b \ast \0} \\
  & = \pair{a \cdot b}{a \ast y \vee (a \ast \0 \ast b)} \vee \pair{a \cdot b}{a \ast (b \ast \0) \vee x \ast b} \\
  & = \pair{a \cdot b}{a \ast y \vee x \ast b} \\
  & = \pair{a}{x} \circ \pair{b}{y},
\end{align*}
  and
\begin{align*}
  \sigma \pair{a}{x} \bs \pair{b}{y} \wedge \pair{a}{x} \bs \gamma \pair{b}{y} & = \pair{a}{a \ast \0} \bs \pair{b}{y} \wedge \pair{a}{x} \bs \pair{\0 \bsresright y}{y} \\
  & = \pair{a \bs b \wedge (a \ast \0) \bsresright y}{a \ast y} \wedge \pair{a \bs (\0 \bsresright y) \wedge x \bsresright y}{a \ast y} \\
  & = \pair{a \bs b \wedge x \bsresright y}{a \ast y} \\
  & = \pair{a}{x} \bs \pair{b}{y}.
\end{align*}
  The last equality is entirely analogous.
\end{proof}

\begin{theorem}[Adjunction for Nagata residuated lattices] \label{thm: nagata rls}
  The (restricted) Nagata product functor going from the \mbox{category} of cyclic pointed residuated $\ell$-bimodules over a residuated lattice to the \mbox{category} of (restricted) Nagata lattices is the right adjoint of the \mbox{structural} bimodule functor. The unit of this adjunction is the embedding ${{m \mapsto \pair{\sigma m}{\gamma m}}}$, the counit is the inverse of the isomorphism $\pair{\embedS}{\embedM}$.
\end{theorem}

\begin{proof}
  Each (restricted) Nagata lattice is a (restricted) Nagata, using Lemma~\ref{lemma: nagata implication} and observing that the equation $x \cdot y = \sigma x \cdot y \vee x \cdot \sigma y$ implies the equation $\gamma (x \cdot y) = (\gamma x \cdot \sigma y) \veegamma (\sigma x \cdot \gamma y)$. Given the adjunction for Nagata structures and the fact that the counit is an isomorphism (and therefore it preserves all existing meets, joins, and residuals), it only remains to prove that the unit $\unit$ preserves meets, joins, and residuals.

  The equations $\sigma(x \vee y) = \sigma x \vee \sigma y$ and $\gamma(x \wedge y) = \gamma x \wedge \gamma y$ ensure that the unit preserves meets and joins. Because $\unit$ preserves multiplication, $\unit(m \bs n) \leq \unit(m) \bs \unit(n)$ and $\unit(m / n) \leq \unit(m) / \unit(n)$. Conversely, $\unit(m \bs n) = \pair{\sigma(m \bs n)}{\gamma(m \bs n)}$ and $\unit(m) \bs \unit(n) = \pair{\sigma m}{\gamma m} \bs \pair{\sigma n}{\gamma n} = \pair{\sigma(\sigma m \bs \sigma n \wedge \sigma(\gamma m \bs \gamma y))}{\sigma m \bs \gamma n}$, which is ensured by the equation involving $\sigma(x \bs y)$.
\end{proof}

\begin{fact}
  The unit of the above adjunction is an iso\-morphism if and only if for each $a \in \algNsigma$ and $x \in \algNgamma$ (such that $a \cdot \0 = \0 \cdot a \leq x$, in the restricted case) there is some $m \in \algN$ such that $\sigma m = a$ and $\gamma m = x$.
\end{fact}

\begin{proof}
  The Nagata product of the structural bimodule consists precisely of the pairs $\pair{a}{x}$ such that $a \ast \0 = \0 \ast a \leq x$ in the structural bimodule, i.e.\ $\gamma(a \cdot \0) = \gamma(\0 \cdot a) \leq x$ in $\algN$, which is equivalent to $a \cdot \0 = \0 \cdot a \leq x$ given that $x \in \algNgamma$. Our condition now states that each such pair is the image of some $m \in \algN$ via $m \mapsto \pair{\sigma m}{\gamma m}$.
\end{proof}

\section{The Nagata equivalence for bilattices and sesquilattices}
\label{sec: bilattices}

  The Nagata adjuction fails to be a categorical equivalence because its unit map $m \mapsto \pair{\sigma m}{\gamma m}$ need not be surjective. However, we can hope to obtain a categorical equivalence by modifying the Nagata adjunction in one of two ways.

  One option is to take inspiration from Sendlewski's categorical equivalence for Nelson algebras~\cite{sendlewski90}. In its basic form, the twist product construction for Nelson algebras has the same defect as the Nagata product construction: a general Nelson algebra is only isomorphic to a subalgebra of a twist product of a Heyting algebra~$\alg{H}$. However, Sendlewski observed that each such subalgebra is determined by a suitable filter $F$ on $\alg{H}$. Thus, each Nelson algebra is isomorphic to the result of a refined twist product construction whose input is the pair $\pair{\alg{H}}{F}$ rather than just~$\alg{H}$. (More precisely, Sendlewski works with the congruence corresponding to the filter $F$.) One could similarly try to describe subalgebras of the restricted Nagata product $\algSmodMzero$ in terms of some additional structure on the $\alg{S}$-bimodule $\alg{M}$. This may well be a direction worth pursuing, but we shall not do so in the current paper.

  The second option is to add further structure to restricted Nagata products such that the presence of this structure will ensure the surjectivity of the unit map $m \mapsto \pair{\sigma m}{\gamma m}$. This structure will be order-theoretic in nature. Namely, in addition to the componentwise partial order
\begin{align*}
  m \leq n & \iff \sigma m \leq \sigma n \text{ and } \gamma m \leq \gamma n,
\end{align*}
  which amounts to
\begin{align*}
  \pair{a}{x} \leq \pair{b}{y} & \iff a \leq b \text{ and } x \leq y,
\end{align*}
  restricted Nagata products also come with the ``twisted'' partial order
\begin{align*}
  m \sqleq n & \iff \sigma n \leq \sigma m \text{ and } \gamma m \leq \gamma n,
\end{align*}
  which amounts to
\begin{align*}
  \pair{a}{x} \sqleq \pair{b}{y} & \iff b \leq a \text{ and } x \leq y,
\end{align*}
  Conversely, the original order can be recovered from the ``twisted'' partial order as
\begin{align*}
  m \leq n \iff \sigma n \sqleq \sigma m \text{ and } \gamma m \sqleq \gamma n.
\end{align*}
  Observe that
\begin{align*}
  & \sigma m \sqleq \sigma n \iff \sigma n \leq \sigma m, & & \gamma m \sqleq \gamma n \iff \gamma m \leq \gamma n.
\end{align*}

  If $\alg{S}$ and $\alg{M}$ are lattices, then the Nagata product $\algSmodM$ is a lattice with respect to the order $\sqleq$, where binary joins and meets are, respectively,
\begin{align*}
  & \pair{a}{x} \oplus \pair{b}{y} \assign \pair{a \wedge b}{x \vee y}, & & \pair{a}{x} \otimes \pair{b}{y} \assign \pair{a \vee b}{x \wedge y}.
\end{align*}
  In other words, $\algSmodM$ then has the structure of a \emph{bilattice}: it comes with two lattice structures such that binary joins and meets in one are isotone with respect to the other. The \emph{restricted} Nagata product $\algSmodMzero$, however, only inherits part of this bilattice structure: for each $\pair{a}{x}, \pair{b}{y} \in \algSmodMzero$ we have $\pair{a}{x} \oplus \pair{b}{y} \in \algSmodMzero$ but not necessarily $\pair{a}{x} \otimes \pair{b}{y} \notin \algSmodMzero$. That is, $\algSmodMzero$ has a structure which is intermediate between that of a lattice and a bilattice. This structure, consisting of a lattice and a join semilattice such that the operations of one are isotone with respect to the other, will be called a \emph{sesquilattice}.

  Observe that each $m \in \algSmodMzero$ decomposes as $m = \sigma m \oplus \gamma m$. In~fact, this sum exists regardless of whether $\alg{S}$ is lattice-ordered.

\begin{fact} \label{fact: sesquilattice decomposition}
  Let $\alg{S}$ be a posemigroup and $\alg{M}$ be a cyclic pointed residuated $\alg{S}$-bimodule. Then for each $m \in \algSmodMzero$
\begin{align*}
  m = \sigma m \oplus \gamma m.
\end{align*}
  Equivalently, if $\embedS a \leq \embedM x$ (i.e.\ if $a \ast \0 \leq x$), then
\begin{align*}
  \pair{a}{x} = \embedS a \oplus \embedM x.
\end{align*} 
\end{fact}

\begin{proof}
  We have $\embedS a \oplus \embedM x = \pair{a}{a \ast \0} \oplus \pair{\0 \bsresright x}{x} = \pair{a \wedge (\0 \bsresright x)}{(a \ast \0) \vee x} = \pair{a}{x}$.
\end{proof}

  The (restricted) Nagata adjunction turns into an equivalence if we expand the single-sorted side of the adjunction by the operations $\oplus$ and $\otimes$ (by the operation~$\oplus$). For the sake of simplicity, we deal only with the maximalist case in the rest of this section. Let us call the expansion of the Nagata product by the operations $\oplus$ and $\otimes$ the \emph{bilattice Nagata product}, and let us call the expansion of the restricted Nagata product by the operation $\oplus$ the \emph{restricted sesquilattice Nagata product}.

\begin{definition}[Nagata bilattices and sesquilattices] \label{def: nagata rsls}
  A \emph{Nagata bilattice} is a Nagata lattice expanded by binary operations $\oplus$ and $\otimes$ which satisfy the equations
\begin{align*}
  \sigma (x \oplus y) & = \sigma (x \wedge y) & \gamma (x \oplus y) & = \gamma (x \vee y) \\
  \sigma (x \otimes y) & = \sigma (x \vee y) & \gamma (x \otimes y) & = \gamma (x \wedge y)
\end{align*}
  A \emph{restricted Nagata sesquilattice} is a restricted Nagata lattice expanded by a binary operation $\oplus$ which satisfies the equations
\begin{align*}
  & \sigma (x \oplus y) = \sigma (x \wedge y), &
  & \gamma (x \oplus y) = \gamma (x \vee y).
\end{align*}
\end{definition}

  It is not immediately obvious that such algebras are indeed bilattices or sesqui\-lattices. However, this will follow from the fact proved below that every Nagata bilattice (every restricted Nagata sesquilattice) is isomorphic to a bilattice Nagata product (a restricted sesquilattice Nagata product).

\begin{fact}
  The (restricted) Nagata product with $\oplus$ of a cyclic pointed residuated lattice-ordered bimodule over a residuated lattice is a Nagata residuated bilattice (a restricted Nagata residuated sesqui\-lattice).
\end{fact}

\begin{proof}
  The equations to be proved are immediate consequences of the definition of $\oplus$ and $\otimes$.
\end{proof}

\begin{theorem}[Equivalence for Nagata bilattices and sesquilattices] \label{thm: nagata rls bilat}
  The bilattice (restricted sesquilattice) Nagata product functor from the \mbox{category} of cyclic pointed residuated lattice-ordered bimodules over a residuated lattice to the \mbox{category} of \mbox{Nagata} bilattices (restricted Nagata sesquilattices) and the \mbox{structural} bimodule functor in the opposite direction form a categorical equivalence. The unit of this equivalence is the isomorphism ${{m \mapsto \pair{\sigma m}{\gamma m}}}$, the counit is the inverse of the isomorphism $\pair{\embedS}{\embedM}$.
\end{theorem}

\begin{proof}
  The four equations in the definition of Nagata bilattices ensure that the unit map preserves both $\oplus$ and $\otimes$:
\begin{align*}
  \pair{\sigma (m \oplus n)}{\gamma (m \oplus n)} = \pair{\sigma(m \wedge n)}{\gamma (m \vee n)} = \pair{\sigma m}{\gamma m} \oplus \pair{\sigma n}{\gamma n},
\end{align*}
  where the second equality holds because $\sigma(m \wedge n)$ is the meet of $\sigma m$ and $\sigma n$ in $\algN_{\sigma}$ and $\gamma (m \vee n)$ is the join of $\gamma m$ and $\gamma n$ in $\algN_{\gamma}$. The same argument shows that in the restricted case, the unit map preserves $\oplus$.

  To prove surjectivity in the case of Nagata bilattices, it suffices to observe that for each $a \in \algNsigma$ and $x \in \algNgamma$ we have $\pair{a}{x} = (m \otimes (m \wedge n)) \vee (n \oplus (m \wedge n))$ for $m \assign \embedS(a)$ and $n \assign \embedM(a)$. To prove surjectivity in the case of restricted Nagata sesquilattices, it suffices to observe that for each $a \in \algNsigma$ and $x \in \algNgamma$ with $a \ast \0 \leq x$ in the structural bimodule, i.e.\ with $\gamma(a \cdot \0) \leq x$ in $\algN$, we have $\pair{a}{x} = \pair{\sigma m}{\gamma m}$ for $m \assign a \oplus x$. This is because $\gamma a = \gamma \sigma a = \gamma (\sigma a \cdot \0) = \gamma(a \cdot \0)$, so $a \leq \gamma a = \gamma(a \cdot \0) \leq x$, so $\sigma(a \oplus x) = \sigma (a \wedge x) = \sigma a = a$ and $\gamma(a \oplus x) = \gamma(a \vee x) = \gamma x = x$.
\end{proof}

\section{Bimodules arising from residuated pairs}

\newcommand{\cdotplus}{\cdot_{\scriptscriptstyle +}}
\newcommand{\cdotminus}{\cdot_{\scriptscriptstyle -}}
\newcommand{\bsminus}{\bs_{\scriptscriptstyle -}}
\newcommand{\sminus}{/_{\scriptscriptstyle -}}

  Let us now make good on our promise to the reader in the introduction and relate the Nagata product construction to twist products. The connection arises when we specialize to Nagata products of bimodules of a particular kind.

  The simplest case of interest arises when a residuated lattice $\algL$ acts on its order dual~$\alg{L}^{\dual}$ by division, i.e.\ by the biaction
\begin{align*}
  a \ast x & \assign x / a, & x \ast a & \assign a \bs x,
\end{align*}
  where $x \in \alg{L}^{\dual}$ and $a \in \algL$. The residuals of this biaction are
\begin{align*}
  x \bsresright y & = x / y, & a \bsresleft x & = x \cdot a,\\
  x \sresleft y & = y \bs x, & x \sresright a & = a \cdot x.
\end{align*}
  This was the biaction considered by Tsinakis \& Wille~\cite{tsinakis+wille06}.

  If $\alg{L}$ is further equipped with a constant $\0$ such that $a \bs \0 = \0 / a$ for each $a \in \alg{L}$, we obtain a cyclic pointed $\alg{L}$-bimodule. The restricted Nagata products of such bimodules were studied by Busaniche et al.~\cite{busaniche+galatos+marcos22}, who provide an analogue of the Nagata adjunction for them. We shall in fact study a more general construction inspired by the non-involutive twist products of Rivieccio \& Spinks~\cite{rivieccio+spinks19}, who consider a set-up consisting of two Heyting algebras linked by a Galois connection.

  Let $\alg{S}_{+}$ be a posemigroup, $\alg{S}_{-}$ be a residuated meet semilattice, and $\lambda\colon \alg{S}_{+} \to \alg{S}_{-}$ be a homomorphism of posemigroups. Each such triple then induces a biaction of $\alg{S}_{+}$ on the order dual $\alg{S}_{-}^{\dual}$ of $\alg{S}_{-}$ by division, which turns it into an $\alg{S}_{+}$-bimodule:
\begin{align*}
  a \ast x & \assign x / \lambda a, & x \ast a & \assign \lambda a \bs x.
\end{align*}
  Suppose further that $\lambda$ has a right adjoint $\rho\colon \alg{S}_{-} \to \alg{S}_{+}$. That is, $\rho$ is an isotone map such that $a \leq \rho \lambda a$ and $\lambda \rho x \leq x$ for each $a \in \alg{S}_{+}$ and $x \in \alg{S}_{-}$, or equivalently
\begin{align*}
  \lambda a \leq x & \iff a \leq \rho x.
\end{align*}
  The above biaction is then residuated, and the residuals are
\begin{align*}
  x \bsresright y & \assign \rho(x / y), & a \bsresleft x & \assign x \cdot \lambda a,\\
  x \sresleft y & \assign \rho(y \bs x), & x \sresright a & \assign \lambda a \cdot x.
\end{align*}
  If $\alg{S}_{+}$ and $\alg{S}_{-}$ are pomonoids and $\lambda$ preserves the unit, then this bimodule is unital. If $\rho$ is a homomorphism of semigroups and $\lambda \circ \rho = \idmap_{\alg{S}_{+}}$, then the multiplication operation $\cdotminus$ of $\alg{S}_{-}$ can be recovered from the multiplication operation $\cdotplus$ of $\alg{S}_{+}$:
\begin{align*}
  x \cdotminus y = \lambda(\rho x \cdotplus \rho y).
\end{align*}
  This leads us to introduce the following notion. We again consider two variants of twist products, a minimalist and a maximalist one.

\begin{definition}[Twistable pairs] \label{def: twist pair}
  A \emph{twistable pair of posemigroups} $\alg{S} \assign \langle \alg{S}_{+}, \alg{S}_{-}, \lambda, \rho \rangle$ consists of:
\begin{itemize}
\item a posemigroup $\alg{S}_{+}$,
\item a meet semilattice and residuated posemigroup $\alg{S}_{-}$, and
\item homomorphisms of posemigroups $\lambda\colon \alg{S}_{+} \to \alg{S}_{-}$ and $\rho\colon \alg{S}_{-} \to \alg{S}_{+}$
\end{itemize}
  such that $\lambda$ is left adjoint to $\rho$ and also left inverse to $\rho$, i.e.\
\begin{align*}
  & \idmap_{\alg{S}_{+}} \leq \rho \circ \lambda, & & \lambda \circ \rho = \idmap_{\alg{S}_{-}}.
\end{align*}
  In a \emph{twistable pair of residuated lattices} we moreover require that $\alg{S}_{+}$ and $\alg{S}_{-}$ be residuated lattices and $\rho$ be a homomorphism of pomonoids. A \emph{(cyclic) pointed} twistable pair is one equipped with a constant $0 \in \alg{S}_{-}$ (with $\lambda a \bs 0 = 0 / \lambda a$ for each $a \in \alg{S}_{+}$). In the bimodule \emph{induced by} a twistable pair, $\alg{S}_{+}$ acts on $\alg{S}_{-}^{\dual}$ as follows:
\begin{align*}
  a \ast x & \assign x / \lambda a, & x \ast a & \assign \lambda a \bs x.
\end{align*}
\end{definition}

  Each twistable pair is a two-sorted ordered algebra. Accordingly, a homomorphism of twistable pairs is a pair of homomorphisms of the positive and negative sorts which commute with $\lambda$ and $\rho$. That is, a homomorphism consists of a posemigroup homomorphism and a residuated posemigroup homomorphism. Note that in a twistable pair of residuated lattices we do not require that the maps $\lambda$ and $\rho$ themselves be homomorphisms of residuated lattices.

  Per the remarks preceding their definition, each (pointed) twistable pair of posemigroups induces a (pointed) residuated bimodule.

\begin{definition}[Twist products] \label{def: twist products}
  The \emph{(restricted) twist product} $\alg{S}^{\bowtie}$ ($\alg{S}^{\bowtie}_{0}$) of a (pointed) twistable pair of posemigroups $\alg{S}$ is the (restricted) Nagata product of the (pointed) residuated bimodule associated with $\alg{S}$ endowed with the additional unary operation
\begin{align*}
  \nagneg \pair{a}{x} \assign \pair{\rho x}{\lambda a}.
\end{align*}
\end{definition}

  More explicitly, the (restricted) Nagata product construction restricts to twistable pairs of posemigroups in the following way: the twist product of a twistable pair of posemigroups $\alg{S} \assign \langle \alg{S}_{+}, \alg{S}_{-}, \lambda, \rho \rangle$ consists of pairs $\pair{a}{x} \in S_{+} \times S_{-}$ with the order
\begin{align*}
  \pair{a}{x} \leq \pair{b}{y} & \iff a \leq b \text{ and } y \leq x
\end{align*}
  and with the multiplication
\begin{align*}
  \pair{a}{x} \circ \pair{b}{y} & \assign \pair{a \cdot b}{y / \lambda a \wedge \lambda b \bs x}.
\end{align*}
  The restricted twist product of a pointed twistable pair of semigroups is obtained by restricting to the following set, which is closed under the operation $\nagneg$:
\begin{align*}
   \set{\pair{a}{x} \in \alg{S}^{\bowtie}}{x \cdot \lambda a \leq \0 \text{ and } \lambda a \cdot x \leq \0}.
\end{align*}
  If $\1$ is a multiplicative unit of $\alg{S}_{+}$, then $\pair{\1}{\0}$ is a multiplicative unit on $\alg{S}^{\bowtie}_{0}$. The conucleus $\sigma$ and a $\sigma$-structural closure operator $\gamma$ on $\alg{S}^{\bowtie}_{0}$ are
\begin{align*}
  & \sigma \pair{a}{x} \assign \pair{a}{\lambda a \bs \0} = \pair{a}{\0 / \lambda a}, & & \gamma \pair{a}{x} \assign \pair{\rho(\0 / x)}{x} = \pair{\rho(x \bs \0)}{x},
\end{align*}
  which yields the following isomorphisms between $\alg{S}_{+}$ and the image of $\sigma$ and between $\alg{S}_{-}^{\dual}$ and the image of $\gamma$:
\begin{align*}
  & \embedSplus\colon a \mapsto \pair{a}{\lambda a \bs \0} = \pair{a}{\0 / \lambda a}, & & \embedSminus\colon x \mapsto \pair{\rho(\0 / x)}{x} = \pair{\rho(x \bs \0)}{x}.
\end{align*}

  The twist product of a twistable pair of residuated $\ell$-semigroups is then a residuated $\ell$-semigroup with the residuals
\begin{align*}
  \pair{a}{x} \bs \pair{b}{y} & \assign \pair{a \bs b \wedge \rho(x / y)}{y \cdot \lambda a}, \\
  \pair{a}{x} / \pair{b}{y} & \assign \pair{a / b \wedge \rho(x \bs y)}{\lambda b \cdot x},
\end{align*}
  and the lattice operations
\begin{align*}
  \pair{a}{x} \wedge \pair{b}{y} & \assign \pair{a \wedge b}{x \vee y}, \\
  \pair{a}{x} \vee \pair{b}{y} & \assign \pair{a \vee b}{x \wedge y}.
\end{align*}
  The restricted twist product of a pointed twistable pair of residuated lattices is a residuated lattice. It is worth observing that in the case of twistable pairs of residuated lattices, $\algS^{\bowtie}_{\0}$ satisfies the equations:
\begin{align*}
  (\nagneg \1) / x = \nagneg x = x \bs (\nagneg \1).
\end{align*}
  That is, the unary operation $\nagneg$ can be replaced by the constant $\nagneg \1$ in the signature of the restricted twist product.

\begin{definition}[Nagata posemigroups with strong negation]
  A \emph{(restricted) Nagata posemigroup with~strong negation} is a (restricted) Nagata posemigroup equipped with a unary operation $\nagneg$ called \emph{strong negation} such that the following equations hold:
\begin{align*}
  x & \leq \nagneg \nagneg x & \nagneg \nagneg \nagneg x = \nagneg x
\end{align*}
\begin{align*}
  \nagneg \sigma \nagneg x & = \gamma x & \gamma( \nagneg x \bsgamma \gamma y) & = \gamma (\gamma x \sgamma \nagneg y) \\
  \nagneg \gamma \nagneg x & = \nagneg \nagneg \sigma x & \gamma( \nagneg \nagneg y \bsgamma \gamma \nagneg x) & = \gamma( \nagneg (x \cdot y)) \\
  \nagneg \nagneg \sigma x & = \sigma \nagneg \nagneg x & (\sigma \nagneg x) \cdot (\sigma \nagneg y) & = \sigma \nagneg (\nagneg y \bsgamma \gamma x)
\end{align*}
  A \emph{(restricted) Nagata lattice with strong negation} is both a (restricted) Nagata lattice and a (restricted) Nagata posemigroup with strong negation. A strong negation is \emph{involutive} if $\nagneg \nagneg x = x$ holds.
\end{definition}

\begin{fact} \label{fact: twist is nagata}
  The (restricted) twist product of each (pointed) twistable pair of posemigroups is a (restricted) Nagata posemigroup with strong negation. The (restricted) twist product of each (pointed) twistable pair of residuated lattices is a (restricted) Nagata residuated lattice with strong negation. It moreover satisfies:
\begin{align*}
  \gamma( \nagneg x \bs y) & = \gamma (x / \nagneg y) \\
  \gamma( \nagneg \nagneg y \bs \nagneg x) & = \gamma( \nagneg (x \cdot y)) \\
  (\sigma \nagneg x) \cdot (\sigma \nagneg y) & = \sigma \nagneg (\nagneg y \bs x)
\end{align*}
\end{fact}

\begin{proof}
  The (restricted) twist product is a particular case of a (restricted) Nagata product, so it suffices to verify that the equations displayed above hold in every (restricted) twist product.

  The (in)equalities which do not feature residuation are easy consequences of the conditions $\idmap_{\alg{S}_{+}} \leq \rho \circ \lambda$ and $\lambda \circ \rho = \idmap_{\alg{S}_{-}}$. The equality $\gamma(\nagneg m \bs n) = \gamma(n / \nagneg m)$ holds because for $m \assign \pair{a}{x}$ and $n \assign \pair{b}{y}$ the second component of both sides of this equality is $y \cdotplus x$. The equality $\gamma(\nagneg \nagneg n \bs \nagneg m) = \gamma \nagneg (m \cdot n)$ holds because for $m \assign \pair{a}{x}$ and $n \assign \pair{b}{y}$ the second component of both sides of this equality are $\lambda a \cdotminus \lambda b = \lambda (a \cdotplus b)$. The equality $(\sigma \nagneg x) \cdot (\sigma \nagneg y) = \sigma \nagneg (\nagneg y \bs x)$ holds because both sides of the equality are in the image of $\sigma$ and for $m \assign \pair{a}{x}$ and $n \assign \pair{b}{y}$ their first component is $\rho x \cdot \rho y = \rho(x \cdot y)$.
\end{proof}

 We already know that every cyclic pointed residuated bimodule can be recovered from as the structural bimodule of its restricted Nagata product. However, we now need to do more, namely to reconstruct the quadruple $\alg{S}$ itself. That is, in addition to recovering the pointed bimodule structure, we need to also reconstruct multiplication and division in $\alg{S}_{-}$ and the maps $\lambda$ and $\rho$. The maps $\lambda$ and $\rho$ are recovered as follows:
\begin{align*}
  \lambda a & \assign \gamma (\nagneg a), & \rho x & = \sigma (\nagneg x).
\end{align*}
  Multiplication and division in $\alg{S}_{-}$ are then recovered as:
\begin{align*}
  x \cdotminus y & \assign \gamma( \nagneg y \bs x), &  x \bsminus y & \assign \gamma(  y \cdot \nagneg x ), & x \sminus y & \assign \gamma( \nagneg y \cdot x).
\end{align*}
  This yields the \emph{untwist functor}, which to each Nagata posemigroup with strong negation $\alg{N}$ assigns the quadruple
\begin{align*}
  \langle \alg{N}_{\sigma}, \alg{N}_{\gamma}^{\dual}, \gamma \circ \nagneg, \sigma \circ \nagneg \rangle,
\end{align*}
  and to each homomorphism $h$ of Nagata posemigroups with strong negation with domain $\alg{N}$ assigns its restriction to $\alg{N}_{\sigma}$ and to $\alg{N}_{\gamma}$.

\begin{fact} \label{fact: untwist}
  The untwisting functor applied to a cyclic pointed Nagata posemigroup with strong negation yields a cyclic pointed twistable pair of posemigroups.
\end{fact}

\begin{proof}
  We already know that it yields a cyclic pointed residuated bimodule. It therefore remains to show that $\alg{S}_{-}^{\dual}$ is a residuated posemigroup, $\lambda$ and $\rho$ are homomorphisms of posemigroups, $\idmap_{\alg{S}_{+}} \leq \rho \circ \lambda$, and $\lambda \circ \rho = \idmap_{\alg{S}_{-}}$.

  The map $\lambda$ is isotone, since $\gamma$ is isotone and $\nagneg$ is antitone. It preserves multiplication, since $\lambda a \cdotminus \lambda b = \gamma(\nagneg \gamma \nagneg b \bs \gamma \nagneg a) = \gamma (\sigma (\nagneg \nagneg b) \bs \gamma \nagneg a) = \gamma \gamma (\nagneg \nagneg b \bs \nagneg a) = \gamma \nagneg (a \cdot b) = \lambda (a \cdotplus b)$, using the equality $\sigma x \bs \gamma y = \gamma (x \bs y)$. Similarly, $\rho$ is isotone, since $\sigma$ is isotone and $\nagneg$ is antitone. It preserves multiplication, since $\rho (x \cdotminus y) = \sigma \nagneg \gamma (\nagneg y \bs x) = \sigma \nagneg \nagneg \sigma \nagneg (\nagneg y \bs x) = \sigma \nagneg (\nagneg y \bs x) = \sigma \nagneg x \cdot \sigma \nagneg y = \rho x \cdotplus \rho x$.

  If $x = \gamma x$, we have $\lambda(\rho(x)) = \gamma \nagneg \sigma \nagneg x = \nagneg \sigma \nagneg \nagneg \sigma \nagneg x = \nagneg \sigma \nagneg x = \gamma x = x$. If $a = \sigma a$, we have $\rho(\lambda(a)) = \sigma \nagneg \gamma \nagneg a = \sigma \nagneg \nagneg \sigma \nagneg \nagneg a = \nagneg \nagneg \sigma a = \nagneg \nagneg a \geq a$.

  The operation $\cdotminus$ is associative, since the operation $\cdotplus$ is associative, being the restriction of $\cdot$ to $\alg{N}_{\sigma}$, and $(x \cdotminus y) \cdotminus z = \lambda \rho ((x \cdotminus y) \cdotminus z) = \lambda ((\rho x \cdotplus \rho y) \cdotplus \rho z = \lambda (\rho x \cdotplus (\rho y \cdotplus \rho z)) = \lambda \rho x \cdotminus (\lambda \rho y \cdotminus \lambda \rho z) = x \cdotminus (y \cdotminus z)$.

  Finally, we show that $\bsminus$ is a residual of $\cdotplus$. The proof for $\sminus$ is analogous. The inequality $x \cdotminus y \leq z$ in $\alg{N}_{\gamma}^{\dual}$ means that $z \leq \gamma (\nagneg y \bs x)$ in $\alg{N}$ for $x, y, z \in \alg{N}_{\gamma}$, while the inequality $y \leq x \bsminus z$ in $\alg{N}_{\gamma}^{\dual}$ means that $\gamma(z \cdot \nagneg x) \leq y$ in $\alg{N}$ for $x, y, z \in \alg{N}_{\gamma}$. But this last inequality implies that $z \cdot \nagneg x \leq y$, so $z \leq y / \nagneg x \leq \gamma(y / \nagneg x) = \gamma(\nagneg y \bs x)$. Conversely, we show that $z \leq \gamma(\nagneg y \bs x)$ implies $\gamma(z \cdot \nagneg ) \leq y$. To this end, let us first observe that for $x, y \in \alg{N}_{\gamma}$ we have $\sigma (\nagneg x) = \nagneg \gamma x = \nagneg x$, so the equality $\gamma x / \sigma y = \gamma(x / y)$ which holds in Nagata posemigroups yields that $y / \nagneg x = \gamma y / \sigma (\nagneg x) = \gamma (y / \sigma (\nagneg x)) = \gamma (y / \nagneg x)$. Thus $z \leq \gamma(\nagneg y \bs x)$ implies that $\gamma(z \cdot \nagneg x) \leq \gamma (\gamma(\nagneg y \bs x) \cdot \nagneg x) = \gamma (\gamma(y / \nagneg x) \cdot \nagneg x) = \gamma( (y / \nagneg x) \cdot \nagneg x) \leq \gamma y = y$.
\end{proof}

\begin{theorem}[Adjunction for Nagata posemigroups with~$\nagneg$] \label{thm: nagata posemigroup with negation adjunction}
  The restricted twist product functor from the \mbox{category} of twistable pairs of posemigroups to the \mbox{category} of Nagata posemigroups with strong negation is right adjoint to the untwist functor. The unit of the adjunction is the map ${m \mapsto \pair{\sigma m}{\gamma m}}$, the counit is the inverse of the iso\-morphism $\pair{\embedSplus}{\embedSminus}$.
\end{theorem}

\begin{proof}
  Given Theorem~\ref{thm: nagata structures} and Facts~\ref{fact: twist is nagata} and~\ref{fact: untwist}, what remains to be verified is that the unit map preserves strong negation and the counit map preserves $\lambda$, $\rho$, $\cdotminus$, $\bsminus$, and $\sminus$. The first claim means that $\pair{\rho \gamma m}{\lambda \sigma m} = \pair{\sigma \nagneg m}{\gamma \nagneg m}$, and indeed $\rho \gamma m = \sigma \nagneg \gamma m$ and $\lambda \sigma m = \gamma \nagneg \sigma m = \nagneg \sigma \nagneg \nagneg  \sigma m = \nagneg \sigma \sigma \nagneg \nagneg m = \nagneg \sigma m = \gamma \nagneg m$.

  The counit map preserves $\lambda$ because $\embedSminus(\lambda a) = \pair{\rho (\lambda a \bs 0)}{\lambda a} = \gamma \pair{\rho(\lambda a \bs 0)}{\lambda a} = \gamma \nagneg \pair{a}{\lambda a \bs 0} = \lambda \embedSplus (a)$. The counit map preserves $\rho$ because $\embedSplus(\rho x) = \pair{\rho x}{\lambda \rho x \bs 0} = \pair{\rho x}{x \bs 0} = \sigma \pair{\rho x}{x \bs 0} = \sigma \nagneg \pair{\rho (x \bs 0)}{x} = \rho \embedSminus(x)$.

  The counit map preserves $\cdotminus$ because
\begin{align*}
  \embedSminus(x \cdotminus y) & = \pair{\rho (x \cdotminus y \bsminus 0)}{x \cdotminus y} 
  = \gamma \pair{\ldots}{x \cdot y} \\
  & = \gamma \pair{\ldots}{x \cdotminus \lambda \rho y}
  =  \gamma(\pair{\rho y}{y \bsminus 0} \bs \pair{\rho(x \bsminus 0)}{x}) \\
  & = \gamma(\nagneg \pair{\rho(y \bsminus 0)}{y} \bs \pair{\rho(x \bsminus 0)}{x})
  = \pair{\rho(x \bsminus 0)}{x} \cdotminus \pair{\rho(y \bsminus 0)}{y} \\
  & = \embedSminus(x) \cdotminus \embedSminus(y),
\end{align*}
  where the ellipsis $\dots$ indicates that the value is irrelevant for the computation.

  The counit map preserves $\bsminus$ because
\begin{align*}
  \embedSminus(x \bsminus y) & = \pair{\rho((x \bsminus y) \bsminus 0)}{x \bsminus y}
  = \gamma \pair{\ldots}{x \bsminus y} \\
  & = \gamma \pair{\ldots}{(x \bs 0) / (y \bs 0) \wedge (x \bsminus y)}
  = \gamma \pair{\ldots}{(x \bs 0) / \lambda \rho (y \bs 0) \wedge (x \bsminus y)} \\
  & = \gamma (\pair{\rho(y \bsminus 0)}{y} \cdot \nagneg \pair{\rho(x \bsminus 0)}{x})
  = \pair{\rho(x \bsminus 0)}{x} \bsminus \pair{\rho(y \bsminus 0)}{y} \\
  & = \embedSminus(x) \bsminus \embedSminus(y),
\end{align*}
  where we use the inequality $x \bsminus y \leq (x \bs 0) / (y \bs 0)$. The proof for $\sminus$ is analogous.
\end{proof}

\begin{theorem}[Adjunction for Nagata lattices with~$\nagneg$] \label{thm: nagata rl with negation adjunction}
  The restricted twist product functor from the \mbox{category} of twistable pairs of residuated lattices to the \mbox{category} of Nagata lattices with strong negation is right adjoint to the untwist functor. The unit of the adjunction is the map ${m \mapsto \pair{\sigma m}{\gamma m}}$, the counit is the inverse of the iso\-morphism $\pair{\embedS}{\embedM}$.
\end{theorem}

\begin{proof}
  This follows from Theorem~\ref{thm: nagata structures} and Theorem~\ref{thm: nagata posemigroup with negation adjunction}.
\end{proof}

\begin{definition}[Nagata bilattices and sesquilattices with strong negation] \label{thm: twist 2}
  A \emph{Nagata bilattice with strong negation} is both a Nagata bilattice and a Nagata lattice with strong negation. A \emph{restricted Nagata sesquilattice with strong negation} is both a restricted Nagata sesquilattice and a Nagata lattice with strong negation.
\end{definition}

  The twist product functor can be expanded to yield a Nagata bilattice as in Section~\ref{sec: bilattices}, and likewise the restricted twist product functor can be expanded to yield a restricted Nagata sesquilattice.

\begin{theorem}[Equivalence for Nagata bilattice and sesquilattices with strong negation] \label{thm: nagata structures with negation}
  The (restricted) twist product functor from the \mbox{category} of (cyclic pointed) twistable pairs of residuated lattices to the \mbox{category} of Nagata bilattices (of restricted Nagata sesquilattices) with strong negation is right adjoint to the untwist functor. The unit of the adjunction is the map ${m \mapsto \pair{\sigma m}{\gamma m}}$, the counit is the inverse of the iso\-morphism $\pair{\embedS}{\embedM}$.
\end{theorem}

\begin{proof}
  This follows from Theorem~\ref{thm: nagata rls} and Theorem~\ref{thm: nagata rl with negation adjunction}.
\end{proof}

  Observe that the strong negation in a (restricted) twist product of a twistable pair of posemigroups is involutive if and only if the maps $\lambda$ and $\rho$ satisfy $\rho \circ \lambda = \idmap_{\alg{S}_{+}}$, i.e.\ if and only if $\lambda$ and $\rho$ are mutually inverse isomorphisms of posemigroups. In that case, the two-sorted structure in effect collapses into a single sort: up to isomorphism we can take $\alg{S}_{-} = \alg{S}_{+}$ and $\lambda = \idmap_{\alg{S}_{+}} = \rho$. This is the case covered by Busaniche et al.~\cite{busaniche+galatos+marcos22}, though their equational axiomatization is of course more parsimonious than ours, being tailored to this particular case.

\section{Twist products and bimonoids of fractions}

  In this final section, we relate the twist product construction and the algebras of fractions introduced in~\cite{galatos+prenosil22}. In particular, we show that the algebra of fractions of a (Boolean-pointed) Brouwerian algebra $\alg{B}$ can be constructed from the (restricted) twist product of $\alg{B}$ as the nucleus image of a conucleus image.

  Let us first explain what we mean by an algebra of fractions. This construction generalizes the Abelian group of fractions of a cancellative commutative monoid to \emph{commutative bimonoids}, defined as posets equipped with two commutative monoidal operations (addition and multiplication) which are isotone in each argument and which satisfy the inequality
\begin{align*}
  x \cdot (y + z) \leq (x \cdot y) + z.
\end{align*}
  Note that multiplication is \emph{not} required to distribute over addition here. An element $y$ of a commutative bimonoid is called a \emph{complement} of $x$ if
\begin{align*}
  & x \cdot y \leq \0, & & \1 \leq x + y,
\end{align*}
  where $\0$ and $\1$ are the monoidal units for addition and multiplication. If $x$ has a complement, it is unique and it will be denoted by $\overline{x}$. A \emph{complemented} commutative bimonoid is one where each element has a complement.

  Complements in bimonoids subsume both multiplicative inverses in monoids and Boolean complements in distributive lattices as special cases. In the former case, we view a monoid as a bimonoid with $x + y \assign x \cdot y$ and $\0 \assign \1$, in which case $\overline{x} = x^{-1}$ if it exists. In the latter case, we view a bounded distributive lattice as a bimonoid with $x \cdot y \assign x \wedge y$ and $x + y \assign x \vee y$, in which case $\overline{x} = \neg x$ if it exists.

\begin{fact}
  Complemented commutative bimonoids (as ordered algebras equipped with the unary operation $x \mapsto \overline{x}$) are term-equivalent to commutative involutive residuated pomonoids: in one direction we take
\begin{align*}
  & \overline{x} \assign x \rightarrow \0, & x + y \assign \overline{\overline{y} \cdot \overline{x}}.
\end{align*}
  In the other direction, we take
\begin{align*}
   x \rightarrow y & \assign \overline{x} + y.
\end{align*}
  in the other, writing $x \rightarrow y$ for $x \bs y = y / x$.
\end{fact}

  Given an embedding $\iota$ of a commutative bimonoid $\alg{A}$ into a complemented commutative bimonoid $\alg{B}$, we say that an element of $\alg{B}$ is a \emph{fraction} if it has the form $\iota x + \overline{\iota y}$ for some $x, y \in \alg{A}$, and we say that it is a \emph{co-fraction} if it has the form $\iota x \cdot \overline{\iota y}$. A \emph{commutative bimonoid of fractions} of a commutative bimonoid $\alg{A}$ is then a complemented commutative bimonoid $\alg{B}$ with an embedding of bimonoids $\iota\colon \alg{A} \to \alg{B}$ such that each element of $\alg{B}$ is a fraction, or equivalently each element of $\alg{B}$ is a co-fraction. Not every commutative bimonoid $\alg{A}$ has such a $\alg{B}$, but if it does, then this $\alg{B}$ is unique up to a unique isomorphism commuting with $\iota$.

  Since the elements of the twist product $\alg{L}^{\bowtie}_{\0}$ may be interpreted as formal fractions over $\algL$, one might hope that the embedding $\iota_{\algL}\colon \pair{a}{a \rightarrow \0}$ will witness $\alg{L}^{\bowtie}_{\0}$ as a commutative bimonoid of fractions of $\algL$ in the technical sense defined above. This fails to hold: not every element of $\alg{L}^{\bowtie}_{\0}$ is a fraction (co-fraction) with respect to~$\embedL$. Nonetheless, we can sometimes identify the commutative bimonoid of fractions of $\alg{L}$ inside $\alg{L}^{\bowtie}_{\0}$. We do so by restricting to elements which are both fractions and co-fractions. The restriction takes place in two steps. We first find a conucleus $\mu$ on $\alg{L}^{\bowtie}_{\0}$ whose image consists precisely of the co-fractions of $\alg{L}^{\bowtie}$. We then find a nucleus $\nu$ on $(\alg{L}^{\bowtie}_{\0})_{\mu}$ whose image consists precisely of the elements of $(\alg{L}^{\bowtie}_{\0})_{\mu}$ which are fractions, i.e.\ of the elements of $\alg{L}^{\bowtie}$ which are both fractions and co-fractions.

  We show that this strategy works out in case $\algL$ is a Boolean-pointed Brouwerian algebra. Recall that a \emph{Brouwerian algebra} is a distributive lattice equipped with a binary operation $x \rightarrow y$ denoting the relative pseudocomplement of $x$ with respect to $y$. Thus bounded Brouwerian algebras are precisely Heyting algebras. A \emph{Boolean-pointed} Brouwerian algebra is one equipped with a constant $\0$ such that the interval $[\0, \1]$ is a Boolean lattice. Equivalently, it is a pointed Brouwerian algebra which satisfies the equation $\intneg \intneg x = x \vee \0$. Here we use the abbreviation $\intneg a \assign a \rightarrow \0$.

  Each pointed Brouwerian algebra may be viewed as a bimonoid where
\begin{align*}
  x \cdot y & \assign x \wedge y, & x + y \assign (\0 \rightarrow (x \wedge y)) \wedge (x \vee y).
\end{align*}
  In particular, if $\0 = \top$, addition reduces to $x + y \assign x \wedge y$. (If the Brouwerian algebra has a bottom element and $\0 = \bot$, addition reduces to $x + y \assign x \vee y$, but such algebras will typically not have a commutative bimonoid of fractions.)

  Let $\alg{B}$ be a Boolean-pointed Brouwerian algebra in the following. It was shown in~\cite{galatos+prenosil22} that $\alg{B}$ has a commutative bimonoid of fractions when viewed as a bimonoid with the above operations. We now provide an alternative construction of this commutative bimonoid of fractions.

  We first prove several lemmas which will help us with the arithmetic of Boolean-pointed Brouwerian algebras.

\begin{lemma}
  In each Boolean-pointed Brouwerian algebra
\begin{align*}
  \text{if $\0 \wedge a = \0 \wedge b$ and $\intneg a = \intneg b$, then $a = b$.}
\end{align*}
\end{lemma}

\begin{proof}
  In each distributive lattice, if $x \wedge a = x \wedge b$ and $x \vee a = x \vee b$, then $a = b$. In particular, $a = b$ whenever $\0 \wedge a = \0 \wedge b$ and $\0 \vee a = \0 \vee b$. Because the interval $[\0, \1]$ is Boolean, the latter equality is equivalent to $\intneg (\0 \vee a) = \intneg (\0 \vee b)$, i.e.\ to $\intneg a = \intneg b$. 
\end{proof}

\begin{lemma}
  In each Boolean-pointed Brouwerian algebra $\intneg a \rightarrow a = a$.
\end{lemma}

\begin{proof}
  We have $\0 \wedge (\intneg a \rightarrow a) = \0 \wedge ((\0 \wedge \intneg a) \rightarrow (\0 \wedge a)) = \0 \wedge (\0 \rightarrow (\0 \wedge a)) = \0 \wedge a$. Clearly $a \leq \intneg a \rightarrow a$ implies $\intneg (\intneg a \rightarrow a) \leq \intneg a$. Conversely, $\intneg a \wedge (\intneg a \rightarrow a) = \intneg a \wedge a \leq \0$, hence $\intneg a \leq \intneg (\intneg a \rightarrow a)$.
\end{proof}

\begin{lemma}
  In each Boolean-pointed Brouwerian algebra
\begin{align*}
  a + b = (\intneg a \rightarrow b) \wedge (\intneg b \rightarrow a).
\end{align*}
\end{lemma}

\begin{proof}
  Let $x = (\intneg a \rightarrow b) \wedge (\intneg b \rightarrow a)$. Then $\0 \wedge (a + b) = \0 \wedge a \wedge b$ and $\0 \wedge x = \0 \wedge a \wedge b$. Also, $(\0 \rightarrow (a \wedge b)) \wedge a \wedge \intneg x = (\0 \rightarrow b) \wedge a \wedge \intneg (\intneg a \rightarrow b) = (\0 \rightarrow b) \wedge a \wedge \intneg (\0 \rightarrow b) \leq \0$, and likewise $(\0 \rightarrow (a \wedge b)) \wedge b \wedge \intneg x \leq \0$, hence $\intneg x \leq \intneg (a+b)$. It remains to prove that $\intneg (a + b) \leq \intneg x$, or equivalently $x \leq \intneg \intneg (a + b) = (a+b) \vee \0$. This is equivalent to the conjunction of $x \leq (\0 \rightarrow (a \wedge b)) \vee \0$ and $x \leq a \vee b \vee \0 = \intneg \intneg (a \vee b)$. But $\0 \wedge x = \0 \wedge a \wedge b \leq a \wedge b$, hence indeed $x \leq (\0 \rightarrow (a \wedge b)) \vee \0$. Finally, $\intneg (a \vee b) \wedge x = \intneg a \wedge \intneg b \wedge (\intneg a \rightarrow b) \wedge (\intneg b \rightarrow a) = \intneg a \wedge \intneg b \wedge b \wedge a \leq \0$, hence indeed $x \leq \intneg \intneg (a \vee b)$.
\end{proof}

\begin{fact}
  The following are equivalent for each pair $\pair{a}{b} \in \alg{B}^{\bowtie}_{\0}$:
\begin{enumerate}[(i)]
\item $\pair{a}{b}$ is a co-fraction of $\alg{B}^{\bowtie}_{\0}$,
\item $\pair{a}{b} = \embedL a \circ \overline{\embedL b}$,
\item $b = a \rightarrow b$.
\end{enumerate}
\end{fact}

\begin{proof}
  If $b = a \rightarrow b$, then $\iota_{\alg{B}}(a) \circ \overline{\iota_{\alg{B}}(b)} = \pair{a \wedge \intneg b}{a \rightarrow b} = \pair{a}{b}$.  Conversely, suppose that $\pair{a}{b} = \iota_{\alg{B}}(x) \circ \overline{\iota_{\alg{B}}(y)} = \pair{x \wedge \intneg y}{x \rightarrow y}$ for some $x, y \in \alg{B}$. Then $a \rightarrow b = (x \wedge \intneg y) \rightarrow (x \rightarrow y) = \intneg y \rightarrow (x \rightarrow y) = x \rightarrow (\intneg y \rightarrow y) = x \rightarrow y = a$.
\end{proof}

\begin{fact}
  The following are equivalent for each pair $\pair{a}{b} \in \alg{B}^{\bowtie}_{\0}$:
\begin{enumerate}[(i)]
\item $\pair{a}{b}$ is a fraction of $\alg{B}^{\bowtie}_{\0}$,
\item $\pair{a}{b} = \embedL a + \overline{\embedL b}$,
\item $a = b \rightarrow a$.
\end{enumerate}
\end{fact}

\begin{proof}
  This follows immediately from the previous fact applied to $\pair{b}{a} = \overline{\pair{a}{b}}$.
\end{proof}

\begin{fact}
  The map $\mu\colon \alg{B}^{\bowtie}_{\0} \to \alg{B}^{\bowtie}_{\0}$ defined as
\begin{align*}
  \mu\colon \pair{a}{b} \mapsto \pair{a}{a \rightarrow b}
\end{align*}
  is a conucleus on $\alg{B}^{\bowtie}_{\0}$ such that $\mu(x \circ y) = \mu(x) \circ \mu(y)$ whose image is the set of all co-fractions of~$\alg{B}^{\bowtie}_{\0}$.
\end{fact}

\begin{proof}
  The map $\mu$ is clearly an interior operator. To prove that it is a conucleus, let us compare $\mu \pair{a}{b} \circ \mu \pair{c}{d}$ and $\mu (\pair{a}{b} \circ \pair{c}{d})$:
\begin{align*}
  \mu \pair{a}{b} \circ \mu \pair{c}{d} & = \pair{a}{a \rightarrow b} \circ \pair{c}{c \rightarrow d} \\ & = \pair{a \wedge c}{(a \rightarrow (c \rightarrow d)) \wedge (c \rightarrow (a \rightarrow b))} \\ & = \pair{a \wedge c}{(a \wedge c) \rightarrow (b \wedge d)},
\end{align*}
\begin{align*}
  \mu (\pair{a}{b} \circ \pair{c}{d}) & = \mu \pair{a \wedge c}{(a \rightarrow d) \wedge (c \rightarrow b)} \\ & = \pair{a \wedge c}{((a \wedge c) \rightarrow (a \rightarrow d)) \wedge ((a \wedge c) \rightarrow (c \rightarrow b))} \\ & = \pair{a \wedge c}{((a \wedge c) \rightarrow d) \wedge ((a \wedge c) \rightarrow b)} \\ & = \pair{a \wedge c}{(a \wedge c) \rightarrow (b \wedge d)}.
\end{align*}
  The claim that the image of this conucleus consists precisely of all co-fractions of $\alg{B}^{\bowtie}_{\0}$ follows immediately from the previous fact.
\end{proof}

\begin{fact}
  The map $\nu\colon (\alg{B}^{\bowtie}_{\0})_{\mu} \to (\alg{B}^{\bowtie}_{\0})_{\mu}$ defined as
\begin{align*}
  \nu\colon \pair{a}{b} \mapsto \pair{b \rightarrow a}{b}
\end{align*}
  is a nucleus on $(\alg{B}^{\bowtie}_{\0})_{\mu}$ whose image is the set of all elements of~$(\alg{B}^{\bowtie}_{\0})_{\mu}$ which are fractions in $\alg{B}^{\bowtie}_{\0}$, i.e.\ all elements of $\alg{B}^{\bowtie}_{\0}$ which are fractions and co-fractions in~$\alg{B}^{\bowtie}_{\0}$.
\end{fact}

\begin{proof}
  If $\pair{a}{b} \in \alg{B}^{\bowtie}_{\0}$, then also $\pair{b \rightarrow a}{b} \in \alg{B}^{\bowtie}_{\0}$, since $(b \rightarrow a) \wedge b = a \wedge b \leq \0$. The pair $\nu \pair{a}{b}$ is a fraction, since $b \rightarrow (b \rightarrow a) = b \rightarrow a$. If $\pair{a}{b}$ is a co-fraction, i.e.\ if $a \rightarrow b = b$, then $\pair{b \rightarrow a}{b}$ is also a co-fraction, since $(b \rightarrow a) \rightarrow b = (b \rightarrow a) \rightarrow (a \rightarrow b) = (a \wedge (b \rightarrow a)) \rightarrow b = a \rightarrow b = b$. The map is clearly a closure operator. To prove that it is a nucleus, let us compare $\nu \pair{a}{b} \circ \nu \pair{c}{d}$ and $\nu(\pair{a}{b} \circ \pair{c}{d})$:
\begin{align*}
  \nu \pair{a}{b} \circ \nu \pair{c}{d} & = \pair{b \rightarrow a}{b} \circ \pair{d \rightarrow c}{d} \\ & = \pair{(b \rightarrow a) \wedge (d \rightarrow c)}{((b \rightarrow a) \rightarrow d) \wedge ((d \rightarrow c) \rightarrow b)} \\ ~ \\
  \nu (\pair{a}{b} \circ \pair{c}{d}) & = \nu \pair{a \wedge c}{(a \rightarrow d) \wedge (c \rightarrow b)} \\ & = \pair{((a \rightarrow d) \wedge (c \rightarrow b)) \rightarrow (a \wedge c)}{(a \rightarrow d) \wedge (c \rightarrow b)}
\end{align*}
  Then the inequality
\begin{align*}
  (b \rightarrow a) \wedge (d \rightarrow c) \leq ((a \rightarrow d) \wedge (c \rightarrow b)) \rightarrow (a \wedge c)
\end{align*}
  is equivalent by residuation to
\begin{align*}
  (b \rightarrow a) (d \rightarrow c) (a \rightarrow d) (c \rightarrow b) \leq a \wedge c,
\end{align*}
  which holds because
\begin{align*}
  (a \rightarrow d) \wedge (d \rightarrow c) \wedge (c \rightarrow b) \wedge (b \rightarrow a) \leq (a \rightarrow b) \wedge (b \rightarrow a) = b \wedge (b \rightarrow a) \leq a, \\
  (c \rightarrow b) \wedge (b \rightarrow a) \wedge (a \rightarrow d) \wedge (d \rightarrow c) \leq (c \rightarrow d) \wedge (d \rightarrow c) = d \wedge (d \rightarrow c) \leq c,
\end{align*}
  using the assumption that $\pair{a}{b}$ and $\pair{c}{d}$ are co-fractions. Similarly, the inequality
\begin{align*}
  (a \rightarrow d) \wedge (c \rightarrow b) \leq ((b \rightarrow a) \rightarrow d) \wedge ((d \rightarrow c) \rightarrow b)
\end{align*}
  holds because by the same assumption
\begin{align*}
  (c \rightarrow b) \wedge (b \rightarrow a) \wedge (a \rightarrow d) & \leq c \rightarrow d \leq d, \\
  (a \rightarrow d) \wedge (d \rightarrow c) \wedge (c \rightarrow b) & \leq a \rightarrow b \leq b. \qedhere
\end{align*}
\end{proof}

\begin{theorem}[Constructing algebras of fractions from twist products]
  $(\alg{B}^{\bowtie}_{\0})_{\mu\nu}$ is a commutative bimonoid of fractions of $\alg{B}$ with the embedding
\begin{align*}
  \embedB\colon a \mapsto \pair{a}{a \rightarrow \0}.
\end{align*}
\end{theorem}

\begin{proof}
  We know that $(\alg{B}^{\bowtie}_{\0})_{\mu\nu}$ is a residuated lattice. To prove that it is a cyclic involutive residuated lattice, observe that the dualizing element of $\alg{B}^{\bowtie}_{\0}$, namely $\pair{\0}{\1}$, is both a fraction and a co-fraction. We show that $\pair{\0}{\1}$ is a dualizing element of~$(\alg{B}^{\bowtie}_{\0})_{\mu\nu}$. The residual $\pair{a}{b} \bs_{\mu\nu} \pair{\0}{\1}$ computed in $(\alg{B}^{\bowtie}_{\0})_{\mu\nu}$ is $\mu(\pair{a}{b} \bs \pair{\0}{\1}) = \mu \pair{(a \rightarrow \0) \wedge (\1 \rightarrow b)}{a \wedge \1} = \mu \pair{b}{a} = \pair{b}{a}$. Similarly, $\mu(\pair{\0}{\1} / \pair{a}{b}) = \pair{b}{a}$.

  Each element of $(\alg{B}^{\bowtie}_{\0})_{\mu\nu}$ is a co-fraction with respect to the map $\embedB$, as required: $\embedB a \circ_{\nu} \overline{\embedB b} = \pair{a}{\intneg a} \circ_{\nu} \pair{\intneg b}{b} = \nu \pair{a \wedge \intneg b}{(a \rightarrow b) \wedge (\intneg b \rightarrow \intneg b)} = \nu \pair{a}{a \rightarrow b} = \nu \pair{a}{b} = \pair{a}{b}$ for each $\pair{a}{b} \in (\alg{B}^{\bowtie}_{\0})_{\mu\nu}$.

  The map $\embedB$ is clearly an order embedding and $\embedB a$ is both a co-fraction and a fraction, since $\intneg a \rightarrow a = a$. We know that $\embedB$ preserves products as a map $\embedB\colon \alg{B} \to \alg{B}^{\bowtie}_{\0}$, therefore it also preserves products as a map $\embedB\colon \alg{B} \to (\alg{B}^{\bowtie}_{\0})_{\mu\nu}$: we have $\embedB a \circ_{\nu} \embedB b = \nu(\embedB a \circ \embedB b) = \nu (\embedB (a \wedge b)) = \embedB(a \wedge b)$. The map also sends the constants $\1$ and $\0$ to $\pair{\1}{\0}$ and $\pair{\0}{\1}$. It remains to show that it preserves addition:
\begin{align*}
  \overline{\embedB a +_{\nu} \embedB b} & = \overline{\embedB a} \circ_{\nu} \overline{\embedB b} \\
  & = \overline{\pair{a}{\intneg a}} \circ_{\nu} \overline{\pair{b}{\intneg b}} \\
  & = \nu (\pair{\intneg a}{a} \circ \pair{\intneg b}{b}) \\
  & = \nu \pair{\intneg a \wedge \intneg b}{(\intneg a \rightarrow b) \wedge (\intneg b \rightarrow a)} \\
  & = \nu \pair{\intneg (a \vee b)}{a + b} \\
  & = \pair{(a+b) \rightarrow \intneg (a \vee b)}{a+b} \\
  & = \pair{\intneg ((a+b) \wedge (a \vee b))}{a + b} \\
  & = \pair{\intneg (a+b)}{a+b}, \\
  \overline{\embedB (a + b)} & = \overline{\pair{a+b}{\intneg (a+b)}} \\
  & = \pair{\intneg (a+b)}{a+b},
\end{align*}
  therefore $\embedB(a+b) = \embedB a + \embedB b$.
\end{proof}

\section*{Acknowledgments}

The work of the first author was funded by the grant 2021 BP 00212 of the grant agency AGAUR of the Generalitat de Catalunya.


\end{document}